\documentclass[reqno,a4paper]{amsart}

\usepackage{amsthm,amsmath,amssymb}
\usepackage{cases}

\usepackage{mathrsfs}
\usepackage{dsfont}

\usepackage{natbib}
\bibpunct{(}{)}{;}{a}{}{,} 

\numberwithin{equation}{section}

\usepackage[colorlinks,citecolor=blue,urlcolor=blue]{hyperref}
\usepackage[shortlabels]{enumitem}

\usepackage[nameinlink]{cleveref}

\newtheorem{theorem}{Theorem}[section]
\newtheorem{lemma}[theorem]{Lemma}
\newtheorem{prop}[theorem]{Proposition}

\theoremstyle{definition}
\newtheorem{definition}[theorem]{Definition}
\newtheorem{remark}[theorem]{Remark}
\newtheorem{condition}[theorem]{Condition}
\newtheorem{example}[theorem]{Example}
\newtheorem{assumption}[theorem]{Assumption}
\newtheorem{nota}[theorem]{Notation}

\Crefname{theorem}{Theorem}{Theorems}
\Crefname{definition}{Definition}{Definitions}
\Crefname{lemma}{Lemma}{Lemmas}
\Crefname{prop}{Proposition}{Propositions}
\Crefname{remark}{Remark}{Remarks}
\Crefname{cor}{Corollary}{Corollaries}
\Crefname{example}{Example}{Examples}
\Crefname{condition}{Condition}{Conditions}
\Crefname{assumption}{Assumption}{Assumptions}
\Crefname{appsec}{Appendix}{Appendices}
\Crefname{appsubsec}{Appendix}{Appendices}

\crefformat{equation}{#2(#1)#3}

\allowdisplaybreaks

\usepackage[dvipsnames]{xcolor}

\usepackage{graphicx}
\usepackage[titletoc,title]{appendix}

\newcommand\donotdisplay[1]{\null}

\newcommand{\ra}{\rightarrow}
\newcommand{\da}{\downarrow}
\newcommand{\ua}{\uparrow}
\renewcommand{\b}{\beta}
\renewcommand{\k}{\kappa}
\newcommand{\E}{\mathord{\rm E}}

\newcommand{\RR}{\mathbb{R}}
\newcommand{\NN}{\mathbb{N}}

\newcommand{\HH}{\mathbb{H}}

\newcommand{\given}{\,|\ }
\renewcommand{\l}{\lambda}

\renewcommand{\th}{\theta}
\renewcommand{\t}{\tau}

\newcommand{\e}{\varepsilon}
\newcommand{\s}{\sigma}
\renewcommand{\d}{\delta}
\renewcommand{\a}{\alpha}
\newcommand{\g}{\gamma}
\newcommand{\q}{\theta}

\newcommand{\Range}{\mathop{\rm Range}\nolimits}
\newcommand{\trace}{\mathop{\rm Trace}\nolimits}

\catcode`\@=11
\newdimen\cdsep
\def\cdstrut{\vrule height .6\cdsep width 0pt depth .4\cdsep}
\def\@cdstrut{{\advance\cdsep by 2em\cdstrut}}

\def\arrow#1#2{
  \ifx d#1
    \llap{$\scriptstyle#2$}\left\downarrow\cdstrut\right.\@cdstrut\fi
  \ifx u#1
    \llap{$\scriptstyle#2$}\left\uparrow\cdstrut\right.\@cdstrut\fi
  \ifx r#1
    \mathop{\hbox to \cdsep{\rightarrowfill}}\limits^{#2}\fi
  \ifx l#1
    \mathop{\hbox to \cdsep{\leftarrowfill}}\limits^{#2}\fi
}
\catcode`\@=12

\PassOptionsToPackage{unicode}{hyperref}
\PassOptionsToPackage{naturalnames}{hyperref}

\usepackage[foot]{amsaddr}

\usepackage{tikz,xcolor,hyperref}

\definecolor{lime}{HTML}{A6CE39}
\DeclareRobustCommand{\orcidicon}{%
	\begin{tikzpicture}
	\draw[lime, fill=lime] (0,0) 
	circle [radius=0.16] 
	node[white] {{\fontfamily{qag}\selectfont \tiny ID}};
	\draw[white, fill=white] (-0.0625,0.095) 
	circle [radius=0.007];
	\end{tikzpicture}
	\hspace{-2mm}
}

\foreach \x in {A, ..., Z}{%
	\expandafter\xdef\csname orcid\x\endcsname{\noexpand\href{https://orcid.org/\csname orcidauthor\x\endcsname}{\noexpand\orcidicon}}
}

\begin{document}

\title[{B}ayesian Bayesian inverse problems]{Bayesian linear inverse problems in regularity scales}

\author[S.~Gugushvili]{Shota Gugushvili$^1$\orcidA{}}
\address{$^1$Biometris\\
	Wageningen University \& Research\\
	Postbus 16\\
	6700 AA Wageningen\\
	The Netherlands}
\email{shota.gugushvili@wur.nl}
\thanks{The research leading to the results in this paper has received funding from the European Research Council under ERC Grant Agreement 320637.}

\author[A. W.~van der Vaart]{Aad W. van der Vaart$^2$\orcidB{}}
\address{$^2$Mathematical Institute\\
	Faculty of Science\\
	Leiden University\\
	P.O. Box 9512\\
	2300 RA Leiden\\
	The Netherlands}
\email{avdvaart@math.leidenuniv.nl}

\author[D. Yan]{Dong Yan$^2$}
\email{d.yan@math.leidenuniv.nl}



\begin{abstract}
	
We obtain rates of contraction of posterior distributions in inverse problems
defined by scales of smoothness classes. We derive abstract results for general
priors, with contraction rates determined by Galerkin approximation. The rate
depends on the amount of prior concentration near the true function and the
prior mass of functions with inferior Galerkin approximation. We apply the general
result to non-conjugate series priors, showing that these priors give 
near optimal and adaptive recovery in some generality, Gaussian priors, and
mixtures of Gaussian priors, where the latter are also shown to be near
optimal and adaptive. The proofs are based on  general testing and
approximation arguments, without explicit calculations on the posterior
distribution. We are thus not restricted to priors based on the singular
value decomposition of the operator. We illustrate the results with
examples of inverse problems resulting from differential equations.
	
\end{abstract}

%\date{\today}

\maketitle

\section{Introduction}
\label{sec:whitenoise_Introduction}

In a statistical inverse problem one observes a noisy version of a transformed signal $Af$ and wishes to 
recover the unknown parameter $f$. In this paper we consider linear inverse problems of the type
\begin{align}
\label{eq:whitenoise_Y=Af+xi}
Y^{(n)} = A f + \frac{1}{\sqrt{n}}\xi,
\end{align}
where $A: H\to G$ is a known bounded linear operator between separable Hilbert spaces $H$ and
$G$, and $\xi$ is a stochastic `noise' process, which is multiplied by the scalar `noise level' $n^{-1/2}$.
The problem is to infer $f$ from the observation $Y^{(n)}$.  To this purpose we assume that the \emph{forward operator} $A$ is injective, but
we shall be interested in the case that the inverse $A^{-1}$, defined on the range of $A$ is not
continuous (or equivalently the range of $A$ is not closed in $G$). The problem of recovering $f$
from $Y^{(n)}$ is then \textit{ill-posed}, and \textit{regularization} 
methods are necessary in order to `invert' the operator $A$. These consist of constructing
an approximation to $A^{-1}$,  with natural properties such as boundedness and whose domain
includes the data $Y^{(n)}$, and applying this to $Y^{(n)}$.
By the discontinuity of the inverse $A^{-1}$, the noise present in the observation is necessarily multiplied,
and regularization is focused on balancing the error in the approximation to $A^{-1}$ to the
size of the magnified noise, in order to obtain a solution that is as close as possible to the true signal $f$.
In this article we study this through the convergence rates of the regularized
solutions to a true parameter $f$, as $n\ra \infty$, i.e.\ as the noise level tends to zero.
In particular, we consider contraction rates of posterior distributions resulting from
a Bayesian approach to the problem. 

There is a rich literature on inverse problems. The case that the noise $\xi$ is a 
bounded \emph{deterministic} perturbation, has been particularly well studied, 
and various general procedures and methods to estimate the convergence rates of regularized solutions 
have been proposed. See the monographs \cite{engl2000regularization,kirsch2011introduction}.
The case of stochastic noise is less studied, but is receiving increasing attention.
In this paper we shall be mostly interested in the case that $\xi$ is white noise
indexed by the Hilbert space $G$, i.e.\ the \emph{isonormal process}, which is characterized by the
requirement that $\langle \xi,w\rangle_G$ is a zero-mean Gaussian variable with
variance $\|w\|_G^2$, for every $w\in G$, where $\langle\cdot,\cdot\rangle_G$ and $\|\cdot\|_G$ are the inner
product and norm in in $G$. Actually the isonormal process cannot be realized as
a Borel-measurable map into $G$, and hence we need to interpret \cref{eq:whitenoise_Y=Af+xi} in
a generalized sense. In our measurement model the observation $Y^{(n)}$ will be a stochastic
process $\bigl(Y^{(n)}(w): w\in G\bigr)$ such that 
\begin{align}
\label{eq:whitenoise_observation_scheme}
Y^{(n)}(w) = \langle Af, w \rangle_{G}+ \frac{1}{\sqrt{n}} \xi(w), \qquad w\in G,
\end{align}
where $\xi=\bigl(\xi(w): w\in G\bigr)$ is the iso-normal process, i.e.\ a zero-mean Gaussian process
with covariance function $\mathbb{E}\bigl(\xi(w_1)\xi(w_2)\bigr) = \langle w_1,w_2 \rangle_{G}$. 
The processes $Y^{(n)}$ and $\xi$ are viewed as measurable maps in the \emph{sample space} $\RR^G$, with its
product $\s$-field. Statistical sufficiency considerations show that
the observation can also be reduced to the vector $\bigl(Y^{(n)}(w_1),Y^{(n)}(w_2),\ldots\bigr)$, which
takes values in the sample space $\RR^\infty$, for any orthonormal basis $(w_i)_{i\in\NN}$ of $G$.
Since in that case the variables $\xi(w_1), \xi(w_2),\ldots$ are stochastically independent standard normal variables,
the coordinates $Y^{(n)}(w_i)$ of this vector are independent random variables 
with normal distributions with means $\langle Af, w_i \rangle_{G}$ and
variance $1/n$. This is known as the \emph{Gaussian sequence model} in statistics, albeit presently
the `drift function' $Af$ involves the operator $A$. See 
\cite{ibragimov2013statistical_estimation,brown1996equivalencenonpararegressionandwhitenoise} and references therein. 

An alternative method to give a rigorous interpretation to white noise $\xi$, is to 
embed $G$ into a bigger space in which $\xi$ can be realized as a Borel measurable map, or to think
of $\xi$ as a cylindrical process. See e.g., \cite{skorohod_1974_integration_in_hilbert_sapce}.
For $G$ a set of functions on an interval, one can also realize $\xi$ as a stochastic integral
relative to Brownian motion, which  takes its values in the `abstract Wiener space' attached to $G$.
We shall not follow these constructions, as they imply the stochastic process
version \cref{eq:whitenoise_observation_scheme}, which is easier to grasp and will be the basis for our proofs.

It is also possible to consider the model \cref{eq:whitenoise_Y=Af+xi} with a noise variable
$\xi$ that takes its values inside the Hilbert space $G$. In this paper we briefly note some results
on this `coloured noise' model, but our main focus is model \cref{eq:whitenoise_observation_scheme}.

The study of statistical (nonparametric) linear inverse problems was initiated by Wahba in 1970s in
\cite{wahba1977integraloperator}. The 1990s paper \cite{donoho1995WVD} used wavelet shrinkage methods,
while around 2000, the authors of \cite{chow_1999_illposed_linear_pde} investigated \cref{eq:whitenoise_Y=Af+xi} 
in the linear partial differential equations setting, while a
systematic study of Gaussian sequence models was presented in \cite{cavalier2002sharp}.
A review of work until 2008 is given in  \cite{cavalier2008nonparametricinverseproblems}.
The connection of regularization methods to the Bayesian approach was recognized 
early on. However, the study of the recovery properties of posterior distributions was started only in
\cite{knapik2011bayesianmild,knapik2013bayesianextreme}. A review of the Bayesian
approach to inverse problems, with many examples, is given in \cite{stuart2010bayesianperspective}.

In the present paper we follow the Bayesian approach. 
This consists of putting a probability measure on $f$, the \emph{prior}, that quantifies one's prior
beliefs on $f$, and next, after collecting the data, updating the prior to the \emph{posterior} measure,
through Bayes' formula. As always, this is the conditional distribution of $f$ given $Y^{(n)}$ in the model,
where $f$ follows the prior measure $\Pi$, a Borel probability distribution on $H$,
and given $f$ the variable $Y^{(n)}$ has the conditional distribution on $\RR^G$ determined by \cref{eq:whitenoise_observation_scheme}.
For a given $f\in H$ the latter conditional distribution is dominated by its distribution under $f=0$. 
The Radon-Nikodym densities $y\mapsto p_f^{(n)}(y)$ of the conditional distributions
can be chosen jointly measurable in $(y,f)$, and by Bayes' formula the posterior distribution 
of $f$ is the Borel measure on $H$ given by
\begin{align}
\label{eq:whitenoise_bayes_formula}
\Pi_n(f \in B \given Y^{(n)}) = \frac{\int_B p_f^{(n)} (Y^{(n)}) \,d\Pi(f) }{\int p_f^{(n)} (Y^{(n)})\, d\Pi(f)}.
\end{align}
The form of the densities $p_f^{(n)}$ is given by the (abstract) Cameron-Martin formula, but will
not be needed in the following (see Lemma~\ref{LemmaKL}).
In the Bayesian paradigm the posterior distribution encompasses all the
necessary information for inference on $f$. An attractive feature of the Bayesian approach 
is that it not only offers an estimation procedure, through a measure of `center' of the
posterior distribution, but also provides a way to conduct uncertainty quantification, through the
spread in the posterior distribution.

One hopes that as the noise level tends to zero, i.e.\ $n\ra\infty$, the posterior measures 
\cref{eq:whitenoise_bayes_formula} will contract to
a Dirac measure at $f_0$ if in reality $Y^{(n)}$ is generated through the model 
\cref{eq:whitenoise_observation_scheme} with $f=f_0$. We shall be interested in the \emph{rate}
of contraction. Following \cite{vandervaart2000posterior,vandervaart2007noniid,vandervaart2017fundamentals}
we say that a sequence $\e_n\downarrow 0$ 
is a rate of posterior contraction to $f_0$ if, for a fixed sufficiently large constant $M$, and $n\ra\infty$,
\begin{align}
\label{eq:whitenoise_contraction_rate_general}
\Pi_n\bigl(f: \|f-f_0\|_H >M \e_n \given  Y^{(n)}\bigr) \overset{\mathbb{P}^{(n)}_{f_0} }{\ra} 0.
\end{align}
We shall use the general approach to establishing such rates of contraction, based on a prior
mass condition and testing condition, explained in \cite{vandervaart2007noniid}.
This was adapted to the inverse setup in \cite{knapiksalomond}, who in a high level
result show how to obtain an inverse rate from 
a rate in the forward problem and a continuity
modulus of the restriction of the operator to suitable sets on which the posterior concentrates.

Much of the existing work on statistical inverse problems is based on
the singular value decomposition (SVD) of the operator $A$; see, e.g.,
\cite{cavalier2008nonparametricinverseproblems}. When $A$ is compact,
the operator $A^*A$, where $A^*$ is the adjoint of $A$, can be
diagonalized with respect to an orthonormal \textit{eigenbasis}, with
eigenvalues tending to zero. The observation $Y^{(n)}$ can then be
reduced to noisy observations on the Fourier coefficients of $Af$ in
the eigenbasis, which are multiples of the Fourier coefficients of
$f$, and the problem is to recover the latter. In the frequentist
setup thresholding or other regularization methods can be applied to
reduce the weight of estimates on coefficients corresponding to
smaller eigenvalues, in which the noise will overpower the signal. In
the Bayesian setup one may design a prior by letting the Fourier
coefficients be (independent) random variables, with smaller variances
for smaller eigenvalues.  These singular value methods have several
disadvantages, as pointed out in
\cite{cohen2004adaptive_galerkin,donoho1995WVD}. First, the eigenbasis
functions might not be easy to compute. Second, and more importantly,
these functions are directly linked to the operator $A$, and need not
be related to the function space (smoothness class) that is thought to
contain the true signal $f$. Consequently, the parameter of interest
$f$ may not have a simple, parsimonious representation in the
eigenbasis expansion, see \cite{donoho1995WVD}. Furthermore, it is
logical to consider the series expansion of the signal $f$ in other
bases than the eigenbasis, for instance, in the situation that one can
only measure noisy coefficients of the signal $f$ in a given basis
expansion, due to a particular experimental setup. See
\cite{goldenshluger2000inverse_functional_hilbertscales,mathe2001optimal_discretization_inverse_hilbert_scales}
for further discussion.

One purpose of the present paper is to work with priors that directly relate to common bases (e.g.,
splines or wavelets bases) and function spaces, rather than to the operator through its singular
value decomposition. We succeed in this aim under the assumption that the operator $A$ 
respects a given scale of function spaces. In \Cref{sec:whitenoise_structure_of_IP}
we first set up such a scale in an abstract manner, and then introduce a smoothing assumption
on the operator $A$ in terms of this scale. Next in 
\Cref{sec:whitenoise_random_series_prior}-\Cref{sec:whitenoise_gaussian_mixture_prior}
we consider priors defined in terms of the scale, rather than the operator. Thus
operator and prior are assumed related, but only indirectly, through the scale.

A canonical example are Sobolev spaces, with the operator $A$ being an
integral operator. This Sobolev space setup with wavelet basis was investigated in
\cite{donoho1995WVD,cohen2004adaptive_galerkin}. In deterministic inverse problems, a more
general setup, considering $A$ that acts along nested Hilbert spaces, \textit{Hilbert scales}, was
initiated by Natterer in \cite{natterer1984hilbertscales} and further developed in, amongst others,
\cite{mathe2001optimal_discretization_inverse_hilbert_scales,mair1996inverse_hilbert_scales,hegland1995hilbertscales_interpolation}.
In the Bayesian context Hilbert scales were used in \cite{florens2016regularizing}, under the assumption that 
the noise $\xi$ is a proper Gaussian element in $G$, and in \cite{agapiou_2013_gaussian_linear_inverse},
but under rather intricate assumptions. 

A second purpose of the present paper is to allow priors that are not necessarily Gaussian.
In the linear inverse problem Gaussian priors are easy, as they lead to Gaussian posterior distributions,
which can be studied by direct means. Most of the results on Bayesian inverse problems
fall in this framework \cite{knapik2011bayesianmild,knapik2013bayesianextreme,florens2016regularizing,agapiou_2013_gaussian_linear_inverse}, exceptions being \cite{ray2013nonconjugate} and
\cite{knapiksalomond}.

Thus in this paper we investigate a Bayesian approach to linear inverse problems that is not based on the SVD 
and does cover non-conjugate, non-Gaussian priors. 

The white noise model represents a limiting case (in an appropriate sense) of the inverse regression model
\begin{align*}
Y_i = A f(x_i) + z_i, \quad i=1,\cdots,n,
\end{align*}
where $z_i$ are independent standard normal random variables. Insights gained in inverse problems in the white noise model shed light on the behaviour of statistical procedures in the inverse regression model, which is the one encountered in actual practice, as the signal $f$ can be typically observed only on a discrete grid of points. It is next at times possible to extend theoretical results obtained in the white noise setting to those in the inverse regression setting.

The paper is organized as follows. In \Cref{sec:whitenoise_structure_of_IP} we introduce in greater
detail our setup along with the assumptions  that will be used in this article. We also present some
examples for illustration. Next we present a general contraction theorem in 
\Cref{sec:whitenoise_general_contraction}, and apply this to two main special cases, series priors and
Gaussian priors in \Cref{sec:whitenoise_random_series_prior}
and~\Cref{sec:whitenoise_gaussian_prior}. The section on Gaussian priors is preceded
by a discussion in \Cref{SectionHilbertScales} of Hilbert scales generated 
by unbounded operators, which next serve as inverse covariance operators.
Since the simple Gaussian prior is not fully adaptive, we introduce Gaussian mixture priors to obtain adaptation in \Cref{sec:whitenoise_gaussian_mixture_prior}. In \Cref{sec:whitenoise_extension}
we discuss several extensions of the present  work. \Cref{SectionProofs} contains the proofs,
and an appendix presents background to some of the tools we need in the proofs.

\begin{nota}
	\label{nota:whitenoise_inequality_up_constant}
	The symbols $\lesssim, \gtrsim, \simeq$ mean $\leq, \geq, =$ up to a positive multiple independent of $n$,
	(or another asymptotic parameter). 
	The constant may be stated explicitly in subscripts, and e.g. $\lesssim_f$ means that it depends on $f$.
\end{nota}

\section{Setup}
\label{sec:whitenoise_structure_of_IP}
In this section we formalize the structure of the inverse problem that will be worked out in this article.

\subsubsection*{Smoothness scales}
The function $f$ in \cref{eq:whitenoise_Y=Af+xi} is an element of a Hilbert space $H$.
We embed this space as the space $H=H_0$ in a `scale of smoothness classes', defined as follows.

\begin{definition}[Smoothness scale]
	\label{cond:smoothness_class_structure}
	For every $s\in \RR$ the space $H_s$ is an infinite-dimensional, separable Hilbert space, with
	inner product $\langle \cdot,\cdot \rangle_{s}$ and induced norm $\|\cdot\|_{s}$.
	The spaces $(H_s)_{s\in\RR}$  satisfy the following conditions:
	\begin{enumerate}[(i)]
		%\item $H_0=H$.
		\item 
		For $s<t$ the space $H_t$ is a dense subspace of $H_s$ and $\|f\|_s \lesssim \|f\|_t$, for  $f\in H_t$. 
		\item 
		For $s\ge 0$ and  $f\in H_0$ viewed as element of $H_{-s}\supset H_0$,
		\begin{equation}
		\label{EqNormDuality}
		\| f\|_{-s}=\sup_{\|g\|_s\le 1}\langle f,g\rangle_0, \qquad f\in H_0.
		\end{equation}
	\end{enumerate}
\end{definition}

The notion of scales of smoothness classes is standard in the literature on inverse problems. 
In the preceding definition we have stripped it to the bare essentials needed in our general result on 
posterior contraction. Concrete examples, as well as more involved structures
such as Hilbert scales, are introduced below.

\begin{remark}
	We may also start with Hilbert spaces $H_s$ for $s\ge 0$ only satisfying  (i) and next define
	$H_{-s}$ for $s\ge0$ to be the dual space $H_s^*$. We  next embed $H_{-s}$ for $s\ge 0$ in $H_0$
	through identifying $H_0$ and its dual $H_0^*\subset H_s^*$ (the restriction of a continuous linear map from
	$H_0$ to $\RR$ to domain $H_s$ is contained in $H_s^*$), and 
	the \emph{norm duality} \cref{EqNormDuality} will be automatic.
	
	It is important that we only `flip'  $H_0$ in this contruction. Every Hilbert space $H_s$ can be identified with
	its dual $H_s^*$ in the usual way, but this involves
	the inner product in $H_s$, and is different from the identification of $H_s^*$ with the `bigger space' $H_{-s}$
	for $s\not=0$.
	
	More generally \cref{EqNormDuality} is implied if, for $s>0$, 
	the space $H_{-s}$ can be identified with the dual space $H_s^*$ of $H_{s}$ and 
	the embedding $\iota: H_0\to H_{-s}$ is the adjoint of the embedding $\iota: H_s\to H_0$,
	after the usual identification of $H_0$ and its dual space $H_0^*$. 
	(The three nested spaces $H_{-s}\supset H_0\supset H_s$ then form a `Gelfand triple'.)
	Indeed, by definition the image $i^*f$ of $f\in H_0=H_0^*$ under the adjoint $\iota^*: H_0^*\to H_s^*$ is the
	map $g\mapsto (\iota^*f)(g)=\langle \iota g, f\rangle_0=\langle g, f\rangle_0$ from $H_s\to \RR$. 
	The norm of this map as an element of $H_s^*$ is $\sup_{\|g\|_s\le 1}(\iota^*f)(g)$. 
	The norm duality follows if $\iota^*f$ is identified with the element $f\in H_0\subset H_{-s}$.
\end{remark}

We assume that the smoothness scale allows good finite-dimensional approximations,
as in the following condition.

\begin{assumption}
	[Approximation]
	\label{AssumptionApproximation}
	For every $j\in\NN$ and $s\in (0,S)$, for some $S>0$, there exists a $(j-1)$-dimensional linear subspace $V_j\subset H_0$ 
	and a number $\d(j,s)$ such that $\d(j,s)\ra0$ as $j\ra\infty$, and such that
	\begin{align}
	\label{eq:approximation_property}
	\inf_{g \in V_j} \|f - g\|_0 &\lesssim \d(j,s)\, \|f\|_s,\\
	\label{eq:stability_property}
	\|g\|_s &\lesssim \frac{1}{\d(j,s)}\,\|g\|_0,\qquad \forall g\in V_j.
	\end{align}
\end{assumption}

This assumption is also common in the literature on inverse problems. The two inequalities \cref{eq:approximation_property} and
\cref{eq:stability_property} are known as inequalities of Jackson and Bernstein type, respectively,
see, e.g., \cite{canuto2010spectral}. The approximation property \cref{eq:approximation_property} 
shows that `smooth elements' $f\in H_s$ are well approximated in $\|\cdot\|_0$ by their projection onto a finite-dimensional
space $V_j$, with approximation error tending to zero as the dimension of $V_j$ tends to infinity. 
Naturally one expects the numbers $\d(j,s)$ that control the approximation to be decreasing in both
$j$ and $s$. In our examples we shall mostly have polynomial dependence $\d(j,s)=j^{-s/d}$, in the case that $H_0$ consists
of functions on a $d$-dimensional domain. 
The stability property \cref{eq:stability_property} quantifies the smoothness norm of the projections 
in terms of the approximation numbers. Both conditions are assumed up to a maximal
order of smoothness $S>0$, and it follows from \cref{eq:stability_property} that $V_j$ must be contained in the space $H_S$.

The approximation property \cref{eq:approximation_property} 
can also be stated in terms of the `approximation numbers' of the canonical embedding
$\iota: H_s \to H_0$. The $j$th \emph{approximation number} of a general bounded linear operator $T: G\to H$ between
normed spaces is defined as 
\begin{align}
\label{eq:whitenoise_approximation_number}
a_j(T :G \to H)
% = \inf_{U: \text{Rank}\, U < j} \|T - U: H_{s+t} \to H_s\| 
=\inf_{U: \text{Rank} U < j} \sup_{f: \|f\|_ G \leq 1} \|(T - U)f\|_H,
\end{align}
where the infimum is taken over all linear operators $U: G \to H$ of rank less than $j$. 
%(The supremum on the far right is the operator norm of the operator $T-U: G\to H$.)
It is immediate from the definitions that the numbers $\d(j,s)$ in \cref{eq:approximation_property} can be taken equal
to the approximation numbers $a_j(\iota: H_s \to H_0)$. 
The set of approximation numbers $a_j(\iota:H_{s+t} \to H_t)$ of the canonical embedding 
describes many characteristics of the smoothness scale $(H_s)_{s\in\RR}$. 
We give a brief discussion in \Cref{sec:whitenoise_approximation}.

%\Cref{AssumptionApproximation} implies that the canonical embedding $\iota: H_s \to H_0$ is a limit of a sequence of finite rank operators, and hence is compact. 

\begin{example}
	[Sobolev classes]
	\label{exa:whitenoise_sobolev_class}
	The most important examples of smoothness classes satisfying \Cref{cond:smoothness_class_structure} 
	are fractional Sobolev spaces on a bounded domain $\mathcal{D}\subset \RR^d$. 
	For a natural number $s\in\NN$ the Sobolev space of order $s$ can be defined by
	\begin{align*}
	H_s(\mathcal{D})=W^{s,2}(\mathcal{D}):= \Bigl\{f \in \mathscr{D}'(\mathcal{D}): \|f\|_s := \sum_{|\a|\leq s} \|D^\a f\|_{L^2(\mathcal{D})} < \infty \Bigr\}.
	\end{align*}
	Here $\mathscr{D}'(\mathcal{D})$ is the space of generalized functions on $\mathcal{D}$ (distributions), i.e.\ the topological dual space of
	the space $C_c^\infty(\mathcal{D})$ of infinitely differentiable functions with compact support in $\mathcal{D}$;
	the sum ranges over the multi-indices $\a =(\a_1,\cdots,\a_d)\in(\{0\} \cup \NN)^d$ 
	with $|\a| := \sum_{i=1}^{s} \a_i\le s$; and $D^\a$ is the differential operator
	\begin{align*}
	D^\a:=\frac{\partial^{\a_1}\partial^{\a_2}\cdots\partial^{\a_d}}{\partial x_1^{\a_1}x_2^{\a_2}\cdots\partial x_d^{\a_d}}.
	\end{align*}
	The definition can be extended to $s \in \RR\backslash \NN$ in several ways. All constructions are 
	equivalent to the Besov space $B_{2,2}^s(\mathcal{D})$, see
	\cite{triebel2010theory1,triebel2008theory4}.  
	
	It is well known that the approximation numbers of the scale of Sobolev spaces satisfy
	\Cref{AssumptionApproximation} with $\d(j,t)= j^{-t/d}$, see
	\cite{triebel2008distributions_sobolevspaces_ellipticeqs}.
\end{example}

\begin{example}
	[Sequence spaces]
	\label{ExampleSequenceSpaces}
	Suppose $(\phi_i)_{i\in\NN}$ is a given orthonormal sequence in  a given Hilbert space $H$, and $1\le b_i\ua\infty$ is
	a given sequence of numbers. For $s\ge 0$, define $H_s$ as the set of all elements $f=\sum_{i\in \NN} f_i \phi_i\in H$
	with $\sum_{i\in \NN}b_i^{2s}f_i^2<\infty$, equipped with the norm 
	$$\|f\|_s=\Bigl(\sum_{i\in \NN}b_i^{2s}f_i^2\Bigr)^{1/2}.$$
	Then $H_0=H$ is embedded in $H_s$, for every $s>0$, and the norms $\|f\|_s$ are increasing in $s$.
	Every space $H_s$ is a Hilbert space; in fact $H_s$ is isometric to $H_0$ under the map $(f_i)\to (f_ib_i^s)$,
	where we have identified the series with their coefficients for simplicity of notation.
	
	For $s<0$, we equip the elements $f=\sum_{i\in\NN} f_i\phi_i$ of $H$, where $(f_i)\in \ell^2$, 
	with the norm as in the display, which is now automatically finite,
	and next define $H_s$ as the metric completion of $H$ under this norm. The space $H_s$ is isometric
	to the set of all sequences $(f_i)_{i\in\NN}$ with $\sum_{i\in\NN}f_i^2b_i^{2s}<\infty$ equipped with the
	norm given on the right hand side of the preceding display, but the
	series $\sum_{i\in\NN}f_i\phi_i$ may not possess a concrete meaning, for instance as a function if $H$ is 
	a function space.
	
	By Parseval's identity the inner product on $H=H_0$ is given by $\langle f, g\rangle_0=\sum_{i\in\NN}f_ig_i$,
	and the norm duality \cref{EqNormDuality} follows with the help of the Cauchy-Schwarz inequality. 
	
	The natural approximation spaces for use in \Cref{AssumptionApproximation} are $V_j=\mathsf{Span}(\phi_i: i<j)$. 
	Inequalities \eqref{eq:approximation_property}-\eqref{eq:stability_property} are satisfied with the
	approximation numbers taken equal to $\d(j,t)= b_j^{-t}$.
\end{example}

The forward operator $A$ in the model \cref{eq:whitenoise_Y=Af+xi} is a bounded linear operator 
$A: H\to G$ between the separable Hilbert spaces $H$ and $G$, and is assumed
to be smoothing. The following assumption makes this precise.
This assumption is satisfied in many examples and is common in the literature
(for instance \cite{cohen2004adaptive_galerkin,natterer1984hilbertscales,goldenshluger2003inversefunctional_hilbertscales}). 

In \Cref{cond:smoothness_class_structure} the space $H$ is embedded
as $H=H_0$ in the smoothness scale $(H_s)_{s\in\RR}$ and hence has norm $\|\cdot\|_0$. 

\begin{assumption}[Smoothing property of $A$]
	\label{cond:ill-posedness-mild}
	For some $\g>0$ the operator $A: H_{-\g} \to G$ is injective and bounded 
	and, for every $f\in H_0$,
	\begin{align}
	\label{eq:isomorphism_s=0}
	\|A f\| \simeq \|f\|_{-\g}.% \tag{W}
	\end{align}
\end{assumption}

\begin{example}[SVD]
	\label{exa:whitenoise_SVD}
	If the operator $A: H\to G$ is compact, then the positive self-adjoint operator $A^*A: H\to H$ possesses
	a countable orthonormal basis of eigenfunctions $\phi_i$, which can be arranged so that 
	the corresponding sequence of eigenvalues $\l_i$ decreases to zero. If $A$ is injective, then
	all eigenvalues, whose roots are known as the \emph{singular values} of $A$, are strictly positive.
	Suppose that there exists $\g>0$ such that 
	\begin{equation}
	\label{EqSingularValuesArePowersOfb}
	\l_i\simeq i^{-2\g}.
	\end{equation}
	If we construct the smoothness classes $(H_s)_{s\in\RR}$ from the basis
	$(\phi_i)_{i\in\NN}$ and the numbers $b_i=i$ as in \Cref{ExampleSequenceSpaces},
	then \cref {eq:isomorphism_s=0} is satisfied.
	
	Indeed, we can write $A$ in polar decomposition as  $Af= U(A^*A)^{1/2}f$, for a partial isometry $U: \Range(A)\to G$,
	and then have $Af=U\sum_i f_i \sqrt{\l_i}\phi_i$, so that $\|Af\|=\|\sum_i f_i i^{-\g}\phi_i\|_0\simeq \|f\|_{-\g}$.
	
	Thus constructions using the singular value decomposition of $A$ can always be accommodated
	in the more general setup described in the preceding.
\end{example}

\begin{example}[Poisson equation]
	\label{exa:poisson_eq}
	The operator $A: L^2(0,1)\to L^2(0,1)$ defined by the differential equation $(Af)''=f$ with Dirichlet
	boundary conditions $Af(0)=Af(1)=0$ is smoothing with $\g=2$ 
	in the Sobolev scale given in Example~\ref{exa:whitenoise_sobolev_class}.
	This is shown in Sections~10.4 and~11.2 in \cite{haase2014functional}.
	%
	%Let $(H_s)_{s\in\RR}$ be the periodic Sobolev spaces of (generalized) functions satisfying the boundary condition $f(0) = f(1) = 0$. Consider the following boundary problem,
	%\begin{align*} u'':= \frac{d^2 u}{dx^2}  = -f, \qquad f\in H_0,\end{align*}
	%with the Dirichlet boundary condition: $u(0) = u(1) = 0$.  The unique solution $u \in H_2$ is given by
	%\begin{align*}u(x) = Af(x) = \int_0^1 k(x,t) f(t)\, dt,\end{align*}
	%where
	%\begin{align*}k(x,t) = \begin{cases}(1-x)t,\quad \text{if } x\geq t,\\(1-t)x,\quad \text{otherwise}.\end{cases}\end{align*}
	%The operator $A: H_0 \to H_0$ is Hilbert-Schmidt and hence compact, and therefore has no bounded inverse. On the other hand the inverse exists as bounded operator $A^{-1}: H_2 \to H_0$, and is given by $A^{-1}f=-f''$. 
	%
	%When $\mathcal{D}\subset \RR$, the Sobolev norm is equivalent to $\|f\|_2 = \|f\|_0 + \|f''\|_0$(Page 217 in \cite{brezis2010functional}). Since $(Af)'' = -f$, we have
	%$\|f\|_0 \leq \|Af\|_2 = \|Af\|_0 + \|(Af)''\|_0 \leq (\|A\| +1)\|f\|_0$, i.e
	%$\|Af\|_2\simeq_A \|f\|_0$. Since the kernel is symmetric, $A$ is self-adjoint. Besides, $A$ is an isomorphism between $H_2$ and $H_0$ as shown above. Hence
	%\begin{align*}
	%\|A f\|_0 = \sup_{\|h\|_0 \leq 1} |\langle h,A f \rangle_0| = \sup_{\|h\|_0 \leq 1} |\langle Ah, f \rangle_0|
	%\simeq_A \sup_{\|h\|_2 \leq 1} |\langle h, f \rangle_0| = \|f\|_{-2},
	%\end{align*}
	%by norm duality argument, for all $f\in H_0$. This shows that \cref{eq:isomorphism_s=0} holds with $\g = 2$.
\end{example}

\begin{example}[Symm's equation \cite{kirsch2011introduction}]
	\label{exa:symm_eq}
	Consider the Laplace equation $\Delta u = 0$ in a bounded set $\Omega \subset \RR^2$ with boundary condition $u = g$ on the boundary $\partial \Omega$.
	The singular layer potential, a boundary integral
	\begin{align*}
	u(x) = -\frac{1}{\pi} \int_{\partial \Omega} h(y) \ln|x - y|\, ds(y), \quad x\in\Omega,
	\end{align*}
	solves the boundary value problem if and only if the density $h$, belonging to the space $C(\partial \Omega)$ of continuous functions on $\partial \Omega$, solves \textit{Symm's equation}
	\begin{align}
	-\frac{1}{\pi}\int_{\partial\Omega} h(y) \ln|x - y|\, ds(y) = g(x), \quad x\in\partial\Omega.
	\end{align}
	Assume the boundary $\partial\Omega$ has a parametrization of the form $\{\rho(s), s\in[0,2\pi]\}$, for some $2\pi$-periodic analytic function $\rho:[0,2\pi] \to \RR^2$ such that $|\dot{\rho}(s)|>0$ for all $s$. Then Symm's equation takes the following form,
	\begin{align*}
	Af(z):= -\frac{1}{\pi}\int_{0}^{2\pi} \log|\rho(z) - \rho(s) |\, f(s)\, ds = g(\rho(z)), \quad z\in[0,2\pi],
	\end{align*}
	where $f(s) = h(\rho(s))|\dot{\rho}(s)|$. It is shown in Theorem 3.18 of 
	\cite{kirsch2011introduction} that  the operator $A$ satisfies \cref{eq:isomorphism_s=0} with $\g = 1$
	and $(H_s)_{s\in\RR}$ the periodic Sobolev spaces on $[0,2\pi]$.
\end{example}

\begin{example} 
	[Radon transform]
	Inverting the Radon transform was recently studied in the Bayesian framework by \cite{MonardNickelPaternain}, who studied
	the  posterior distribution of smooth functionals for general Gaussian priors, but not the inversion of 
	the whole function. The SVD of the transform is known
	(see \cite{Xu1,Xu2}) and can be used to put the problem in our framework, in the spirit
	of Example~\ref{exa:whitenoise_SVD}. This would give a Bayesian parallel to the rate results in \cite{Silverman}. 
	We do not know if other standard smoothness scales could be used within our framework as well.
\end{example}

\begin{remark}
	\label{RemarkAAnul}
	For all our purposes the smoothing condition \eqref{eq:isomorphism_s=0} can be relaxed to 
	\eqref{EqSmoothingConditionRelaxed}-\eqref{EqSmoothingConditionRelaxedtwo}.
	This relaxation covers the situation where there exists an operator $A_0$ that satisfies
	\eqref{eq:isomorphism_s=0} and is a `version' of $A$ in that the two operators possess a common inverse,
	such as when $A$ and $A_0$ are defined to solve a differential equation with
	different boundary conditions. Lemma~\ref{lem:whitenoise_R_j_estimates_modified} shows that
	the relaxed version of the smoothing condition is then satisfied by the map
	$f\mapsto [Af]$ of $f$ in the class of $Af$ in the quotient space $G/R(A-A_0)$.
\end{remark}

\section{General Result}
\label{sec:whitenoise_general_contraction}
In this section we present a general theorem on posterior contraction. We form the posterior distribution
$\Pi_n(\cdot\given Y^{(n)})$ as in \cref{eq:whitenoise_bayes_formula}, given a prior $\Pi$ on the space $H=H_0$ and 
an observation $Y^{(n)}$, whose conditional distribution given $f$ is determined
by the model \cref{eq:whitenoise_observation_scheme}. We study this random
distribution under the assumption that $Y^{(n)}$ follows the model 
\cref{eq:whitenoise_observation_scheme} for a given `true' function $f=f_0$, which we assume
to be an element of $H_\b$ in a given smoothness scale $(H_s)_{s\in\RR}$,
as in \Cref{cond:smoothness_class_structure}. 

The result is based on an extension of the testing approach of \cite{vandervaart2017fundamentals} 
to the inverse problem \cref{eq:whitenoise_observation_scheme}. It resembles the approach in 
\cite{NicklSohl,Trabs,nickl2018JEMS, ray2013nonconjugate}
or \cite{knapiksalomond}, except that the inverse problem is handled with the help of the Galerkin method, which is a well known strategy
in numerical analysis to solve the operator equation $y = Af$ for $f$, in particular for differential and integral operators. 
The Galerkin method has several variants, which are useful depending on the properties of the operator involved. 
Here we use the least squares method, which is of general application;
for other variants and background, see e.g., \cite{kirsch2011introduction}. 
In \Cref{sec:whitenoise_galerkin_method} we give a self-contained derivation of the necessary inequalities,
exactly in our framework. We note that the Galerkin method only appears as
a tool to state and derive a posterior contraction rate. In our context it does not enter
into the solution of the inverse problem, which is achieved through the Bayesian method.

Let $W_j=AV_j\subset G$ be the image under $A$ of a finite-dimensional approximation space  $V_j$ linked to the 
smoothness scale $(H_s)_{s\in\RR}$ as in \Cref{AssumptionApproximation},
and let $Q_j: G\to W_j$ be the orthogonal projection onto $W_j$. 
If $A: H\to G$ is injective, then $A$ is a bijection between the finite-dimensional
vector spaces $V_j$ and $W_j$, and hence for every $f\in H$ there exists $f^{(j)}\in V_j$ such
that $Af^{(j)}=Q_j Af$. The element $f^{(j)}$ is called the \emph{Galerkin solution} to $Af$ 
in $V_j$. By the projection theorem in Hilbert spaces it is characterized by the property that 
$f^{(j)}\in V_j$ together with the orthogonality relations
\begin{equation}
\label{EqGalerkinOrthogonality}
\langle A f^{(j)}, w\rangle=\langle Af, w\rangle, \qquad w \in W_j.
\end{equation}
The idea of the Galerkin inversion is to project the (complex) object $Af$ onto the
finite-dimensional space $W_j$, and next find the inverse image $f^{(j)}$ of the projection, in the
finite-dimensional space $V_j$, as in the diagram:
$$\begin{matrix}
H_0\ni f&\arrow{r} A& Af\in G\\
&                  &\arrow{d} {Q_j}\\
V_j\ni f^{(j)}&\arrow{l}{A^{-1}}&Q_jAf\in W_j\\
\end{matrix}$$
Clearly the Galerkin solution to an element $f\in V_j$ is $f$ itself, but in general $f^{(j)}$ is
an approximation to $f$, which will be better for increasing $j$, but increasingly complex.
The following theorem uses a dimension $j=j_n$ that balances approximation to complexity,
where the complexity is implicitly determined by a testing criterion.

\begin{theorem}
	\label{thm:whitenoise_general_contraction_rate_subspace}
	For smoothness classes $(H_s)_{s\in\RR}$ as in \Cref{cond:smoothness_class_structure}, 
	assume that $\|Af \|\simeq \|f\|_{-\g}$ for some $\g>0$, 
	and let $f^{(j)}$ denote the Galerkin solution to $Af$ relative to linear subspaces $V_j$ associated to $(H_s)_{s\in\RR}$
	as in \Cref{AssumptionApproximation}. Let $f_0\in H_\b$ for some $\b\in(0,S)$,
	and for $\eta_n\ge \e_n\downarrow 0$ such that $n\e_n^2\ra\infty$,
	and $j_n\in \NN$ such that $j_n\ra\infty$, and some $c>0$, assume
	\begin{align}
	j_n&\le c n\e_n^2,\label{EqjEpsilon}\\
	\eta_n&\ge \frac{\e_n}{\d(j_n,\g)},\label{EqEtaEpsilonDelta}\\
	\eta_n&\ge \d(j_n,\b). \label{EqEtaDelta}%  \|f_0^{(j_n)}-f_0\|_0.
	\end{align}
	Consider prior probability distributions  $\Pi$ on $H_0$ satisfying
	\begin{align}
	\Pi\bigl(f: \|Af-Af_0\|<\e_n\bigr)&\ge e^{-n\e_n^2},\label{EqPriorMass}\\
	\Pi\bigl(f: \|f^{(j_n)}-f\|_0> \eta_n\bigr)&\le e^{-4n\e_n^2}.\label{EqPriorApproximation}
	\end{align}
	Then the posterior distribution in the model \cref{eq:whitenoise_observation_scheme}
	contracts at the rate $\eta_n$ at $f_0$, i.e.\ for a sufficiently large constant $M$ we have 
	$\Pi_n\bigl(f: \|f-f_0\|_0> M\eta_n\given Y^{(n)}\bigr)\ra 0$,
	in probability under the law of $Y^{(n)}$ given by \cref{eq:whitenoise_observation_scheme} with $f=f_0$.
\end{theorem}

\begin{proof}
	The Kullback-Leibler divergence and variation between the distributions of $Y^{(n)}$ under two functions
	$f$ and $f_0$ are given by $n\|Af-Af_0\|^2/2$ and twice this quantity, respectively.
	(At a referee's request, a proof is provided in Lemma~\ref{LemmaKL}.) Therefore the
	neighbourhoods $B_{n,2}(f_0,\e)$ in (8.19) of \cite{vandervaart2017fundamentals} contain the ball $\{f\in H_0:  \|Af-Af_0\|\le \e\}$.
	By assumption \cref{EqPriorMass} this has prior mass at least $e^{-n\e_n^2}$.
	
	Because the quotient of the left sides of \cref{EqPriorMass} and \cref{EqPriorApproximation}
	is $o(e^{-2n\e_n^2})$, the posterior probability of the set $\bigl\{f: \|f^{(j_n)}-f\|_0>\eta_n\bigr\}$ tends to zero, by
	Theorem~8.20 in \cite{vandervaart2017fundamentals}.
	
	By a variation of Theorem~8.22 in \cite{vandervaart2017fundamentals} it is now sufficient
	to show the existence of tests $\t_n$ such that, for some $M>0$,
	$$P_{f_0}^{(n)}\t_n\ra0,\qquad
	\sup_{\substack{f: \|f-f_0\|_0>M\eta_n,\\ \|f^{(j_n)}-f\|_0\le \eta_n}}P_f^{(n)}(1-\t_n)\le e^{-4n\e_n^2}.$$
	Indeed, in the case that the prior mass condition (8.20) in Theorem~8.22 of \cite{vandervaart2017fundamentals} can be strengthened to
	(8.22), as is the case in our setup in view of \cref{EqPriorMass}, it suffices to verify (8.24) only for a single value of $j$. 
	Furthermore, we can apply Theorem~8.22 with the metrics $d_n(f,g)=\|f-g\|_0\,\e_n/\eta_n$ in order
	to reduce the restriction $d_n(\th,\th_{n,0})>M\e_n$ to $\|f-f_0\|_0>M\eta_n$. 
	
	Fix any orthonormal basis $(\bar\psi_i)_{i<j}$ of $W_j=AV_j$ and define
	\begin{align*}\bar Y_j=\sum_{i<j} Y^{(n)}_{\bar\psi_i}\bar\psi_i
	&=\sum_{i<j} \langle Af, \bar\psi_i\rangle\bar\psi_i+\frac1{\sqrt n}\sum_{i<j} \xi_{\bar\psi_i}\bar\psi_i\\
	&= Q_jAf +\frac1{\sqrt n}\bar \xi_{j},
	\end{align*}
	where $\bar \xi_j:=\sum_{i<j} \xi_{\bar\psi_i}\bar\psi_i$. The latter is ``standard normal in the finite-dimensional
	space $W_j$'': because $(\xi_{\bar\psi_i})_{i<j}$ are i.i.d.\ standard normal variables,
	the variable $\langle \bar \xi_j, w\rangle=
	\sum_{i<j}\xi_{\bar\psi_i}\langle \bar\psi_i,w\rangle$ is $N\bigl(0, \|Q_jw\|^2\bigr)$-distributed, 
	for every $w\in G$.  
	
	Let the operator $R_j: G\mapsto V_j$ be defined as $R_j=A^{-1}Q_j$, where $A^{-1}$ is the inverse of $A$, which is 
	well defined on the range $W_j=AV_j$ of $Q_j$. Then by definition $R_jAf$ is equal to the Galerkin solution $f^{(j)}$ to $Af$.
	By the preceding display $R_j \bar Y_j$ is a well-defined Gaussian random element in $V_j$, satisfying
	\begin{equation}
	R_j \bar Y_j=f^{(j)}+\frac1{\sqrt n}R_j \bar \xi_j.
	\label{EqRY}
	\end{equation}
	The variable $R_j \bar \xi_j$ is a Gaussian random element in $V_j$ with strong and weak
	second moments
	\begin{align*}
	\E\bigl\|R_j \bar \xi_j\bigr\|_0^2&\le \|R_j\|^2\E\|\bar\xi_j\|^2
	=\|R_j\|^2\E\sum_{i<j} \xi_{\bar\psi_i}^2=\|R_j\|^2(j-1)\lesssim \frac j{\d(j,\g)^2},\\
	\sup_{\|f\|_0\le 1}\E \langle R_j \bar \xi_j, f\rangle_0^2
	&=\!\sup_{\|f\|_0\le 1}\E \langle \bar \xi_j, R_j^*f\rangle^2
	=\! \sup_{\|f\|_0\le 1}\| Q_jR_j^*f\|^2\le \|R_j^*\|^2\lesssim \frac1{\d(j,\g)^2}.
	\end{align*}
	In both cases the  inequality on $\|R_j\|=\|R_j^*\|$ at the far right side follows from 
	\cref{eq:whitenoise_galerkin_R_j_estimate}.
	
	The first inequality implies that the first moment 
	$\E\bigl\|R_j \bar \xi_j\bigr\|_0$ of the variable $\|R_j\bar\xi_j\|_0$ is bounded
	above by $\sqrt j/\d(j,\g)$. 
	By Borell's inequality (e.g. Lemma~3.1 in \cite{LedouxTalagrand} and subsequent discussion),
	applied to the Gaussian random variable $R_j\bar\xi_j$ in $H_0$, we see that
	there exist positive constants $a$ and $b$ such that, for every $t>0$,
	$$\Pr\Bigl(\|R_j\bar\xi_j\|_0>t+a\frac{\sqrt j}{\d(j,\g)}\Bigr)\le e^{-bt^2\d(j,\g)^2}.$$
	For $t=2\sqrt n \eta_n/\sqrt b$ and $\eta_n$, $\e_n$ and $j_n$ satisfying
	\cref{EqjEpsilon},  \cref{EqEtaEpsilonDelta} and \cref{EqEtaDelta} this yields,
	for some $a_1>0$,
	\begin{equation}
	\label{EqBorellSpecial}
	\Pr\Bigl(\|R_{j_n}\bar\xi_{j_n}\|_0>a_1\sqrt n\eta_n\Bigr)\le e^{-4n\e_n^2}.
	\end{equation}
	We apply this to bound the error probabilities of the tests 
	\begin{equation}
	\label{EqTestTaun} 
	\tau_n=1\bigl\{\|R_{j_n}\bar Y_{j_n}-f_0\|_0\ge M_0 \eta_n\bigr\},
	\end{equation}
	where $M_0$ is a given constant, to be determined.
	
	Under $f_0$, the decomposition \cref{EqRY} is valid with $f=f_0$,
	and hence $R_j\bar Y_j-f_0=n^{-1/2}R_j\bar \xi_j+f_0^{(j)}-f_0$. By the triangle inequality
	it follows that $\t_n=1$ implies that $n^{-1/2}\|R_{j_n}\bar \xi_{j_n}\|_0\ge M_0\eta_n-\|f_0^{(j)}-f_0\|_0$.
	By \cref{eq:whitenoise_galerkin_solution_error_estimate}
	the assumption that  $f_0\in H_\b$ implies that $\|f_0^{(j)}-f_0\|_0\le M_1 \d(j,\b)$, for some $M_1$,
	which at $j=j_n$ is further bounded by $M_1\eta_n$, by assumption \cref{EqEtaDelta}. 
	Hence the probability of an error of the first kind satisfies
	\begin{align*}
	P_{f_0}^{(n)} \tau_n%&\le P_{f_0}^{(n)}\Bigl(\frac1{\sqrt n}\|R_{j_n}\bar\xi_{j_n}\|\ge M_0\eta_n-\|f_0^{(j)}-f_0\|_0\Bigr)\\
	&\le \Pr\Bigl(\frac1{\sqrt n}\|R_{j_n}\bar\xi_{j_n}\|_0\ge (M_0-M_1)\eta_n\Bigr).
	\end{align*}
	For $M_0-M_1>a_1$, the right side is bounded by $e^{-4n\e_n^2}$, by \cref{EqBorellSpecial}.
	
	Under $f$ the decomposition \cref{EqRY} gives that 
	$R_j\bar Y_j-f_0=n^{-1/2}R_j\bar \xi_j+f^{(j)}-f_0$. By the triangle inequality $\t_n=0$ implies
	that $n^{-1/2}\|R_{j_n}\bar\xi_{j_n}\|_0\ge\|f^{(j_n)}-f_0\|_0- M_0\eta_n$. 
	For $f$ such that $\|f-f_0\|_0>M\eta_n$ and $\|f-f^{(j_n)}\|_0\le \eta_n$, we have
	$\|f^{(j_n)}-f_0\|\ge (M-1)\eta_n$. Hence the probability of an error of the second kind satisfies
	\begin{align*}
	P_{f}^{(n)}(1- \tau_n) %&\le P_{f}^{(n)}\Bigl(\frac1{\sqrt n}\|R_{j_n}\bar\xi_{j_n}\|\ge\|f^{(j)}-f_0\|- M_0\eta_n\Bigr)\\
	&\le \Pr\Bigl(\frac1{\sqrt n}\|R_{j_n}\bar\xi_{j_n}\|_0\ge(M-1-M_0)\eta_n\Bigr),
	\end{align*}
	For $M-1-M_0>a_1$, this is bounded by $e^{-4n\e_n^2}$, by \cref{EqBorellSpecial}.
	
	We can first choose $M_0$ large enough so that $M_0-M_1>a_1$,
	and next $M$ large enough so that $M-1-M_0>a_1$, to finish the proof.
\end{proof}

Inequality \cref{EqPriorMass} is the usual \emph{prior mass condition} for the `direct problem' of
estimating $Af$  (see \cite{vandervaart2000posterior}). It determines the rate of contraction $\e_n$ of the posterior
distribution of $Af$ to $Af_0$. The rate of contraction $\eta_n$ of the posterior distribution of $f$ is slower
due to the necessity of (implicitly) inverting the operator $A$. The theorem shows that the rate $\eta_n$
depends on the combination of the prior, through \cref{EqPriorApproximation}, and the
inverse problem, through the various approximation rates.

\begin{remark}
	It would be possible to obtain the theorem as a corollary of Theorem 2.1 in
	\cite{knapiksalomond}. We would take the sets ${\mathcal S}_n^c$ in the latter high-level result equal to the sets
	$\{f: \|f^{(j_n)}-f\|_0>\eta_n\}$ appearing in \cref{EqPriorApproximation}. To verify the
	conditions of \cite{knapiksalomond} for this choice, most of the preceding proof would be needed.
	Since the next theorem appears not to be a consequence of this approach, and its proof uses
	the preceding proof, we have given a direct proof instead.
\end{remark}

The theorem applies to a true function $f_0$ that is `smooth' of order $\b$ (i.e., $f_0\in H_\b$).
For a prior that is constructed to give an optimal contraction rate for multiple
values of $\b$ simultaneously, the theorem may not give the best result.
The following theorem refines Theorem~\ref{thm:whitenoise_general_contraction_rate_subspace} by considering a mixture prior of the
form 
\begin{equation}
\label{EqMixturePrior}
\Pi=\int \Pi_\t\,dQ(\t),
\end{equation}
where $\Pi_\t$ is a prior on $H$, for every given `hyperparameter' $\t$ running through some measurable space, 
and $Q$ is a prior on this hyperparameter. The idea is to  \emph{adapt} the prior
to multiple smoothness levels through the hyperparameter $\t$.

\begin{theorem}
	\label{TheoremGeneralAdaptation}
	Consider the setup and assumptions of \Cref{thm:whitenoise_general_contraction_rate_subspace} with
	a prior of the form \cref{EqMixturePrior}. Assume that
	\cref{EqjEpsilon}, \cref{EqEtaEpsilonDelta}, \cref{EqEtaDelta} and \cref{EqPriorMass} hold,
	but replace \cref{EqPriorApproximation} by the pair of conditions, 
	for numbers $\eta_{n,\t}$ and $C>0$ and every $\t$,
	\begin{align}
	\Pi_\t\bigl(f: \|f^{(j_n)}-f\|_0> \eta_{n,\t}\bigr)&\le e^{-4n\e_n^2},\label{EqPriorApproximationExtended}\\
	\Pi_\t\bigl(f: \|f-f_0\|_0<2\eta_{n,\t}\bigr)&\le e^{-4n\e_n^2}, \quad\forall \t \text{ with }\eta_{n,\t}\ge C\eta_n.\label{EqPriorMassRough}
	\end{align}
	Then the posterior distribution in the model \cref{eq:whitenoise_observation_scheme}
	contracts at the rate $\eta_n$ at $f_0$, i.e.\ for a sufficiently large constant $M$ we have 
	$\Pi_n\bigl(f: \|f-f_0\|_0> M\eta_n\given Y^{(n)}\bigr)\ra 0$,
	in probability under the law of $Y^{(n)}$ given by \cref{eq:whitenoise_observation_scheme} with $f=f_0$. 
\end{theorem}

\begin{proof}
	We take the parameter of the model as the pair $(f,\t)$, which receives the joint prior
	given by $f\given \t\sim \Pi_\t$ and $\t\sim Q$. With abuse of notation, we denote this prior also by $\Pi$.
	The likelihood still depends on $f$ only, but the joint prior gives rise to a posterior distribution on 
	the pair $(f,\t)$, which we also denote by $\Pi_n(\cdot\given Y^{(n)})$, by a similar abuse of notation.
	
	By \cref{EqMixturePrior} and \cref{EqPriorApproximationExtended}--\cref{EqPriorMassRough},
	\begin{align*}
	\Pi\bigl((f,\t):  \|f^{(j_n)}-f\|_0> \eta_{n,\t}\bigr)
	&%=\int \Pi_\t\bigl((f,\t): \|f^{(j_n)}-f\|_0\ge \eta_{n,\t}\bigr)\,dQ(\t)
	\le e^{-4n\e_n^2},\\
	\Pi\bigl((f,\t): \eta_{n,\t}\ge C\eta_n, \|f-f_0\|_0< 2\eta_{n,\t}\bigr)
	&%=\int \Pi_\t\bigl((f,\t): \eta_{n,\t}\ge c\eta_n, \|f-f_0\|_0\le \eta_{n,\t}\bigr)\,dQ(\t)
	\le e^{-4n\e_n^2}.
	\end{align*}
	In view of \cref{EqPriorMass} and Theorem~8.20 in \cite{vandervaart2017fundamentals}, the posterior probabilities of
	the two sets in the left sides tend to zero. As in the proof of \Cref{thm:whitenoise_general_contraction_rate_subspace},
	we can apply a variation of Theorem~8.22 in \cite{vandervaart2017fundamentals} to see that it is now sufficient
	to show the existence of tests $\t_n$ such that, for some $M\ge 2C$,
	$$P_{f_0}^{(n)}\t_n\ra0,\qquad
	\sup_{\substack{(f,\t):\|f-f_0\|_0>M\eta_n\vee 2\eta_{n,\t},\\ \|f^{(j_n)}-f\|_0\le \eta_{n,\t}}}P_f^{(n)}(1-\t_n)\le e^{-4n\e_n^2}.$$
	%(Note that for a given pair $(f,\t)$ we have $\|f-f_0\|_0>M\eta_n\vee 2\eta_{n,\t}$ if and only if
	%($\eta_{n,\t}<C\eta_n$ and $\|f-f_0\|_0>M\eta_n$) or ($\eta_{n,\t}\ge C\eta_n$ and $\|f-f_0\|_0>M\eta_n$ and $\|f-f_0\|_0>2\eta_N{n,\t}$),
	%where we use $M\ge 2C$ to simplify the first.)
	(Note that $M\eta_n\vee 2\eta_{n,\t}=M\eta_n$ if $\eta_{n,\t}<C\eta_n$ and $M\ge 2C$.)
	We use the tests defined in \cref{EqTestTaun}, as in the proof of \Cref{thm:whitenoise_general_contraction_rate_subspace}.
	The latter proof shows that the tests are consistent. We adapt the bound on the power, as follows.
	
	By the triangle inequality $\t_n=0$ implies that,
	for $(f,\t)$ with $\|f-f_0\|_0>M\eta_n\vee 2\eta_{n,\t}$ and $\|f^{(j_n)}-f\|_0\le \eta_{n,\t}$,
	\begin{align*}
	n^{-1/2}\|R_{j_n}\bar\xi_{j_n}\|_0&\ge\|f^{(j_n)}-f_0\|_0- M_0\eta_n
	\ge \|f-f_0\|_0-\|f^{(j_n)}-f\|_0-M_0\eta_n\\
	&\ge M\eta_n\vee 2\eta_{n,\t}-\eta_{n,\t}-M_0\eta_n\ge (M/2-M_0)\eta_n.\end{align*}
	Hence by \cref{EqBorellSpecial} the probability of an error of the second kind 
	is bounded by $e^{-4n\e_n^2}$, for $M$ sufficiently large that $M/2-M_0>a_1$.
\end{proof}

In a typical application of the preceding theorem the priors $\Pi_\t$ for $\t$ such that $\eta_{n,\t}\ge C\eta_n$
will be the priors on `rough' functions, with `intrinsic' contraction rate $\eta_{n,\t}$
slower than $\eta_n$. These `bad' priors do not destroy the overall contraction rate,
because they put little mass near the true function $f_0$, by
condition \cref{EqPriorMassRough}. It is necessary to address these priors explicitly in the conditions, because 
they will typically fail the approximation condition \cref{EqPriorApproximation}, which must be 
relaxed to \cref{EqPriorApproximationExtended}. A further generalization 
might be to allow the truncation levels $j_n$ to depend on $\t$, but this will not be needed 
for our examples.

Inspection of the proof shows that the posterior probability of the
sets $\{\t: \eta_{n,\t}\gtrsim C\eta_n\}$ tends to zero. This means that the posterior correctly
disposes of the models that are `too rough', for the given true function $f_0$. 
In general there is no similar protection against models that are too smooth, but this
does not affect the contraction rate.

\section{Random Series Priors}
\label{sec:whitenoise_random_series_prior}
Suppose that $\{\phi_i\}_{i\in\NN}$ is an orthonormal basis of $H=H_0$ that gives optimal approximation
relative to the scale of smoothness classes $(H_s)_{s\in\RR}$ in the sense
that the linear spaces $V_j = \mathsf{Span}\{\phi_i\}_{i<j}$ satisfy \Cref{AssumptionApproximation}. 
Consider a prior defined as the law of the random series 
\begin{align}
\label{eq:whitenoise_random_series_prior}
f= \sum_{i=1}^M f_i \phi_i,
\end{align}
where $M$ is a random variable in $\NN$ independent from the independent random variables 
$f_1,f_2,\ldots$ in $\RR$.

\begin{condition}[Random series prior]
	\label{cond:whitenoise_random_series_prior}
	\begin{enumerate}[(i)]
		\item The probability density function $p_M$ of $M$ satisfies, for  some positive constants $b_1,b_2$,
		\begin{align*}
		e^{-b_1 k}\lesssim p_M(k) \lesssim e^{-b_2 k}, \qquad \forall k\in\NN.
		\end{align*}
		\item The variable $f_i$ has density $p(\cdot/\k_i)/\k_i$, for a given probability density
		$p$ on $\RR$ and a constant $\k_i>0$ such that, for some $C>0$ and $w>0$, $\a, \b_0>0$,
		\begin{align}
		\label{eq:whitenoise_random_series_prior_tail_weight}
		p(x) &\gtrsim e^{-C|x|^w},\\
		i^{-\b_0/d} (\log i) ^{-1/w}&\lesssim \k_i \lesssim i^{\a}.
		\label{EqBoundsScalingKappa}
		\end{align}
	\end{enumerate}
\end{condition}

Priors of this type were studied in \cite{rousseau_2013_sieve_priors, ray2013nonconjugate},
and applied to inverse problems in the SVD framework in \cite{ray2013nonconjugate}
(see Section~3.1 of the latter paper for discussion).
For Gaussian variables $f_j$ and degenerate $M$ the series \cref{eq:whitenoise_random_series_prior}
is a Gaussian process, and has been more widely studied, but  we focus here on the non-Gaussian case.
Since the basis $(\phi_i)_{i\in\NN}$ used in the prior is linked to the smoothness class $(H_s)_{s\in\RR}$,
rather than to the operator $A$, the prior is not restricted to the SVD framework. Of course,
in the theorem below we do require the operator to be smoothing in the same smoothness scale,
thus maintaining a link between prior and operator.

The assumption on the density $p_M$ is mild and is satisfied, for instance, by the Poisson distribution. 
The assumption on the density $p$ is mild as well, and is satisfied by many distributions with full support in $\RR$, 
including the Gaussian and Laplace distributions. The parameter $\b_0$ in \cref{EqBoundsScalingKappa}
must be a lower bound on the smoothness of the true parameter
$f_0$. Apart from this, condition \cref{EqBoundsScalingKappa} is also very mild, and allows the
scale parameters $\k_i$ to tend both to zero or to infinity.

The preceding random series prior is not conjugate to the inverse problem \cref{eq:whitenoise_Y=Af+xi}.
In general the resulting posterior distribution will not have a closed form expression,
but must be computed using simulation, such as Markov chain Monte Carlo,
or approximated using an optimisation method, such as variational approximation.
However, the contraction rate of the posterior distribution can be established without the help of
an explicit expression for the posterior distribution, as shown in the following theorem. 

\begin{theorem}[Random Series Prior]
	\label{thm:whitenoise_random_series_prior_mild}
	Let $(\phi_i)_{i\in \NN}$ be an orthonormal basis of $H_0$ such that the spaces
	$V_j = \mathsf{Span}\{\phi_i\}_{i<j}$ satisfy \Cref{AssumptionApproximation} with 
	$\d(j,s) = j^{-s/d}$ relative to smoothness classes $(H_s)_{s\in\RR}$ as in \Cref{cond:smoothness_class_structure}.
	Assume that $\|Af \|\simeq \|f\|_{-\g}$ for some $\g>0$, and let $f_0\in H_\b$ for some $\b\in (0,S)$.
	Then, for the random series prior defined in \cref{eq:whitenoise_random_series_prior} 
	and satisfying \Cref{cond:whitenoise_random_series_prior} with $\b_0 \leq \b$,
	and sufficiently large $\underline M>0$, for $\t=(\b + \g )(1+2\g/d)/(2\b + 2 \g + d)$,
	\begin{align*}
	\Pi_n\Bigl( f: \|f - f_0\|_0 > \underline M  n^{-\b/(2\b + 2 \g + d)} (\log n)^\t \given Y^{(n)}\Bigr) 
	\overset{\mathbb{P}^{(n)}_{f_0} }{\ra} 0.
	\end{align*}
\end{theorem}

The rate $n^{-\b/(2\b + 2 \g + d)}$ is known to be the minimax rate of estimation of a $\b$-regular function
on a $d$-dimensional domain, in an inverse problem with inverse parameter $\g$ (see, e.g., \cite{cohen2004adaptive_galerkin}).
The assumption that $\d(j,s) = j^{-s/d}$ places the setup of the theorem in this setting,
and hence the rate of contraction obtained in the preceding theorem is the minimax rate
up to a logarithmic factor. The rate is adaptive to the regularity of $\b$ of the
true parameter, which is not used in the construction of the prior, apart from the assumption that
$\b\ge\b_0$.  (See \cite{GhosalLemberandvanderVaart} and 
Chapter~10 in \cite{vandervaart2017fundamentals} for general discussion of adaptation in the Bayesian sense.)

The proof of the theorem is deferred to Section~\ref{SectionProofs}; it will be
based on \Cref{thm:whitenoise_general_contraction_rate_subspace}.
%The theorem can be viewed as an extension of Proposition 3.2 in \cite{ray2013nonconjugate}, in which the prior is constructed in the SVD framework.

\begin{example}
	[Wavelet basis]
	\label{exa:whitenoise_random_series_prior}
	Let $p$ be a standard normal density, $p_M$ a standard Poisson probability mass function,
	and set the scaling parameters $\k_i$ equal to 1 (no scaling). 
	
	Consider an $S$-regular orthonormal wavelet basis $\{\phi_{j,k}\}$ for the 
	space of square-integrable functions on the $d$-dimensional torus $(0,2\pi]^d$.
	We can renumber the index $(j,k)$ into $\NN$ by ordering the basis functions by their multiresolution levels, 
	$2^{jd}+k$, and next construct the random series prior \cref{eq:whitenoise_random_series_prior}. 
	
	An $S$-regular orthonormal wavelet basis is known to correspond to the scale of Sobolev
	spaces up to smoothness level $S$. Therefore, by \Cref{thm:whitenoise_random_series_prior_mild}, 
	the contraction rate of the posterior distribution is 
	$n^{-\b/(2\b + 2 \g + d)}$ times a logarithmic factor whenever the operator is smoothing
	relative to the Sobolev scale and the true function $f_0$ belongs
	to the Sobolev space of order $\b$, for $\b_0\le \b< S$. Thus the posterior distributions are adaptive
	up to a logarithmic factor to the scale of Sobolev spaces of orders between $\b_0$ and $S$.
	
	For increasing $\b\ge S$ the rate given by the theorem still improves. However, the `regularity' $\b$ defined by the
	scale $(H_s)_{s\in\RR}$ may then not coincide with the Sobolev scale.
\end{example}

\section{Hilbert Scales}
\label{SectionHilbertScales}
A \textit{Hilbert scale} is a special type of smoothness scale $(H_s)_{s\in\RR}$, as in \Cref{cond:smoothness_class_structure},
generated by an unbounded operator. Such a scale is particularly useful in connection
to differential operators and Gaussian priors, as considered in the next sections. 
For reference we include a short summary on Hilbert scales, and some examples.
Extended discussions of Hilbert scales in the context of regularization theory can be
found e.g. in Chapter 8 of \cite{engl2000regularization}, and a general treatment of the subject
in \cite{krein1966banachscales}.

A Hilbert scale is generated by an unbounded operator $L: D(L)\subset H_0\to H_0$,
with domain  $D(L)$ such that 
\begin{enumerate}[(i)]
	\item[(a)] $D(L)$ is dense in $H_0$, (i.e.\ `$L$ is densely defined'),
	\item[(b)] $D(L)=D(L^*)$,
	\item[(c)] $\langle Lx,y \rangle = \langle x,Ly \rangle$ for all $x,y \in D(L)$, (i.e.\ `$L$ is symmetric'),
	\item[(d)] $\langle Lx,x \rangle\geq \k\|x\|^2$, for all $x\in D(L)$, and some $\k>0$.
\end{enumerate}
The set $D(L^*)$ in (b) is the domain of the adjoint $L^*$ of $L$, which is \emph{defined} as 
the set of all $y\in H$ such that the map $x\mapsto \langle Lx,y\rangle$ from $D(L)$ to $\RR$ is continuous. 
Thus $D(L^*)$ depends on the domain $D(L)$, which is considered part of the definition of $L$ and
is restricted by (a) only. Together, requirements (b) and (c) are equivalent to the requirement that $L$ 
be \emph{self-adjoint}. The latter is important for the existence of a spectral decomposition, used below.

The domain of the $k$-th power of the operator $L$ is defined, by induction for $k=2,3,\ldots$, as
$D(L^k) = \bigl\{ f\in D(L^{k-1}): Lf \in D(L) \bigr\}$,  (with $L^1=L$).
All powers $L^k$, for $k\in \NN$, are defined on 
\begin{align}
\label{eq:whitenoise_intersection_of_domains_of_L}
H_\infty := \bigcap_{k\in \NN} D(L^k).
\end{align}
It can be shown that $H_\infty$ is dense in $H_0$ (Lemma 8.17 in \cite{engl2000regularization}). 
Next, using spectral theory, fractional powers $L^s$ can be defined as well on the domain $H_\infty$,
for every $s\in \RR$, through integration with respect to the spectral family $(E_\lambda)$ of $L$, i.e.
\begin{align*}
L^s:= \int_\RR \lambda^s\, d E_\lambda = \int_{\k}^{\infty} \lambda^s\, d E_\lambda.
%\quad \text{with } D(L^s) = \left\{ f\in (H_\infty)^*: \|f\|_{H_s}^2 :=\int_\RR \lambda^{2s} d|E_\lambda f|^2 < \infty \right\},
\end{align*} 
This allows to define an inner product on $H_\infty$ by, for $h,g \in H_\infty$ and $s\in \RR$,
\begin{align}
\label{eq:whitenoise_hilbert_scale_innerproduct}
\langle h,g \rangle_s := \langle L^s h, L^s g \rangle.
\end{align}

\begin{definition}[Hilbert scales]
	\label{def:whitenoise_hilbert_scale}
	The Hilbert space $H_s$ is the completion of $H_\infty$ with respect to the norm induced by the
	inner product $\langle \cdot,\cdot \rangle_s$ defined in \cref{eq:whitenoise_hilbert_scale_innerproduct}. 
	The family $(H_s)_{s\in\RR}$ is called the \textit{Hilbert scale generated by} $L$.
\end{definition} 

The following proposition, adapted from Proposition~8.19 in \cite{engl2000regularization},
lists basic properties of Hilbert scales. 

\begin{prop}
	\label{prop:whitenoise_hilbert_scale_properties}
	Let $L$ be a densely defined unbounded operator satisfying (a)--(d). Then the 
	Hilbert scale $(H_s)_{s\in \RR}$ is a smoothness scale in the sense of \Cref{cond:smoothness_class_structure}, with
	\begin{enumerate}[(i)]
		\item[(i)] $\|f\|_s \leq \k^{s-t}\|f\|_t$, for $f\in H_t$, and $s<t$. 
		\item[(ii)] $\|f\|_s \leq \|f\|_r^{\lambda} \|f\|_t^{1 - \lambda}$, for $\lambda = (t-s)/(t-r)$, and $r < s < t$.
	\end{enumerate}
	Furthermore, for any $s,t\in\RR$ the operator $L^{t-s}$ has a unique extension from $H_\infty$ to 
	a bounded, self-adjoint operator $L^{t-s}: H_t\to H_s$, satisfying
	\begin{enumerate}[(i)]
		\item[(iii)] $\|L^{t-s}f\|_s\simeq\|f\|_t$, for $f\in H_t$.
		\item[(iv)] $L^{t-s} = L^tL^{-s}$.
		\item[(v)] $(L^s)^{-1} = L^{-s}$.
	\end{enumerate}
\end{prop}

Somewhat abusing notation, we have denoted the extension of $L^{t-s}$ in the proposition using the same symbol $L^{t-s}$.
Taking $s=0$ or $t=0$, we see that $L^s: H_s\to H_0$ and $L^{s}: H_0\to H_{-s}$ are norm isomorphisms, for every $s\in\RR$.
In particular, the unbounded densely defined operator $L: D(L)\subset H_0\to H_0$ that generates the scale can be extended to a
bounded operator $L: H_1\to H_0$, by strengthening the norm on its domain, and also to a bounded operator
$L: H_0\to H_{-1}$, by extending its range space and weakening the norm of its range space.
Moreover, the inverse map is a norm isomorphism $L^{-1}: H_0\to H_1$, and hence is certainly 
bounded as an operator $L^{-1}: H_0\to H_0$.

The eigenvalues of $L^{-1}$ are closely connected to the approximation numbers 
in \Cref{AssumptionApproximation}.

\begin{prop}
	\label{PropositionEigenvaluesApproximationNumbers}
	If $L^{-1}: H_0\to H_0$ is compact with eigenvalues $\l_j\da0$, then
	\Cref{AssumptionApproximation} is satisfied in the Hilbert scale $(H_s)_{s\in\RR}$ generated by $L$,
	with $\d(j,t)\simeq \l_j^{t}$ and $S=\infty$. In fact, there exist linear spaces $V_j$ of dimension $j-1$ such that,
	for $s\ge 0$ and $t\in\RR$,
	\begin{align}
	\label{EqApproximationPropertyExtended}
	\inf_{g \in V_j} \|f - g\|_t &\lesssim \d(j,s)\, \|f\|_{s+t},\\
	\label{EqStabilityPropertyExtended}
	\|g\|_{s+t} &\lesssim \frac{1}{\d(j,s)}\,\|g\|_t,\qquad \forall g\in V_j.
	\end{align}
\end{prop}

\begin{proof}
	Because $L^{-1}: H_0\to H_0$ is compact, there exists an orthonormal basis $(\phi_i)_{i\in \NN}$ of eigenfunctions in $H_0$.
	It may be checked that $f=\sum_{i\in \NN} f_i\phi_i$ has $L^sf=\sum_{i\in \NN} f_i\l_i^{-s}\phi_i$,  and
	square norm $\|f\|_s^2=\sum_{i\in \NN} f_i^2\l_i^{-2s}$, provided the latter series converges. Take $V_j$ equal to the linear span of the
	first $j-1$ eigenfunctions. Then $f-P_jf=\sum_{i\ge j}f_i\phi_i$ and hence
	$\|f-P_jf\|_t^2=\sum_{i\ge j} f_i^2\l_i^{-2t}\le \l_j^{2s}\sum_{i\ge j} f_i^2\l_i^{-2t-2s}\le  \l_j^{2s}\|f\|_{s+t}^2$, for $s,t\ge 0$,
	and for $f\in V_j$ we have $\|f\|_{s+t}^2=\sum_{i< j} f_i^2\l_i^{-2s-2t}\le \l_j^{-2s}\sum_{i\le j} f_i^2\l_i^{-2t}=\l_j^{-2s}\|f\|_{t}^2$.
\end{proof}

The sequence spaces of Example~\ref{ExampleSequenceSpaces} are one class of examples of Hilbert scales, generated by
the operator $L: (f_i)\mapsto (f_ib_i)$. More intricate Hilbert scales arise from (elliptic) differential operators.
These are useful in that they can incorporate boundary conditions, which are then automatically inherited by a Gaussian 
prior attached to such a scale. The following one-dimensional example is simplistic, but illustrative.

\begin{example}[Sobolev scales]
	%(Cf.\ \cite{grubb2010distributions_and_operators}.)
	\label{exa:hs_sobolev_laplacian}
	Consider the one-dimensional negative Laplacian
	$$-\Delta = -\frac{d^2}{d x^2}$$
	as an operator on the space $C_c^\infty (0,1)$ of infinitely often differentiable functions with
	compact support in $(0,1)$, viewed as subset of $L^2(0,1)$, with range space $L^2(0,1)$.  
	On this domain this operator
	is not self-adjoint, but it has a self-adjoint extension (with differentiation interpreted in the
	sense of distributions) to the space of all functions
	$f\in W^{2,2}(0,1)$ satisfying the \emph{Dirichlet boundary condition}
	\begin{align}
	\label{eq:whitenoise_dirichlet_bc}
	f(0) = 0=f(1).
	\end{align}
	(See Theorem 4.23 in \cite{grubb2010distributions_and_operators}.) 
	The eigenfunctions of the Laplacian under the Dirichlet boundary 
	condition are the functions $x\mapsto \sin(j\pi x)$, for  $j\in \NN$, with eigenvalues of
	the order $b_j\asymp j^{-1}$. The corresponding Hilbert
	scale can also be described as the sequence space generated by this orthogonal basis.
	
	Because the Laplacian is a second derivative it is natural to half the scale parameter, or
	equivalently use the root negative Laplacian $L := \sqrt{-\Delta}$ as the generator of the scale
	(where the root is defined through the spectral decomposition). 
	
	The boundary conditions play an important role in defining the scale. Technically they are needed to 
	create a domain on which the operator is self-adjoint. An alternative choice to the Dirichlet is the 
	\emph{Cauchy boundary condition}
	\begin{align*}
	f'(0)=0=f(1).
	\end{align*}
	This leads to the sequence scale generated by the eigenfunctions $x\mapsto\cos ((j-1/2)\pi x)$,
	for $j\in \NN$,  and is different from the Dirichlet scale. 
	Again the eigenvalues of $L^{-1}$ are of the order $j^{-1}$.
	
	Incidentally, it is shown in \cite{Neubauer} that the full Sobolev scale ($s\in\RR$) of Example~\ref{exa:whitenoise_sobolev_class} is 
	not a Hilbert scale for any generating operator $L$. Also in that sense the boundary conditions are essential.
\end{example}

\begin{example}[Abel operator]
	For a given kernel function $K: (0,1)\times(0,1)\to \RR$ and $\a\in(0,1]$, consider the operator 
	$A: L^2(0,1)\to L^2(0,1)$ given by 
	$$Af(x)=\frac1{\Gamma(\a)}\int_0^x (x-s)^{\a-1}K(x,s)f(s)\,ds.$$
	For $K=1$ this gives the classical \emph{Abel operator}.
	Under mild smoothness conditions on $K$, it is shown in \cite{Gorenflo}, Theorem~1, that $A$ is smoothing 
	(i.e.\ \cref{eq:isomorphism_s=0} holds) of order $\g = 1$
	for the Sobolev scale generated by the root negative Laplacian under the Cauchy boundary condition,
	described in Example~\ref{exa:hs_sobolev_laplacian}.
\end{example}

While in the preceding examples the boundary consists of just two points, for 
multi-dimensional domains the boundary is continuous, 
and the restrictions of functions in a smoothness class 
to the boundary form an infinite-dimensional function space.
By choosing an appropriate generating operator, we can construct a Hilbert scale of functions
that automatically satisfy a desired boundary condition.

Consider a  second order elliptic differential operator
$L: D(L)\subset L^2(\mathcal{D})\to L^2(\mathcal{D})$ 
on a bounded domain $\mathcal{D}\subset \RR^d$ with a smooth boundary. 
To generate a Hilbert scale the operator must be self-adjoint, which involves
both the form of the operator and its domain $D(L)$, where 
different domains will lead to different Hilbert scales.
Self-adjointness requires both the structural property 
$\int_ {\mathcal{D}}(Lf) g\,d\l= \int_{\mathcal{D}} f(Lg)\,d\l$ and 
equality of the domains of $L$ and its adjoint $L^*$, where the domain of $L^*$ is 
by definition the set of $g$ such that the left side of the preceding equality is a continuous
function of $f\in D(L)\subset L^2(\mathcal{D})$. The latter implies restrictions on the domain,
which are typically revealed through partial integrations.

One may start from $L$ as an operator on the space of $C^\infty$-functions with support
within $\mathcal{D}$. The closure of this operator (defined by the
closure of its graph $\{(f, Lf): f\in C_c^\infty(\mathcal{D})\}$ 
in $L^2(\mathcal{D})\times L^2(\mathcal{D})$),
is known as the \emph{minimal realization} associated with $L$,  while the \emph{maximal realization}
has domain of definition 
$\bigl\{f \in L^2(\mathcal{D}): \exists u \in L^2(\mathcal{D})\text{ such that } Lf= u\text{   weakly} \bigr\}$.
(See Definitions~4.1-4.2 in \cite{grubb2010distributions_and_operators};  
a self-adjoint operator is always closed, which explains 'minimal'.) Neither of these operators need to be self-adjoint,
but there always exist self-adjoint operators with a domain between these two extremes.
% In this section we also use
%$$W_0^{k,2} (\mathcal{D}) = \bigl\{f \in W^{k,2} (\mathcal{D}): D^\alpha f|_{\partial\mathcal{D} } = 0,|\alpha|\leq k -1 \bigr\}, \qquad k=1,2,\cdots,$$
%which is the closure of $C_c^\infty(\mathcal{D})$ under the $W^{k,2}$-norm.

For example, the minimal domain of the $d$-dimensional Laplacian operator $L=-\Delta$ is given by
$\{f\in  W^{2,2} (\mathcal{D}): f|_{\partial\mathcal{D} }= \nabla f|_{\partial\mathcal{D} }=0\}$ 
(see Theorem 10.19 in \cite{schmuedgen2012unbounded_selfadjoint_operator_on_hilbert_space})
and the maximal domain contains the full Sobolev space $W^{2,2}(\mathcal{D})$ given in 
\Cref{exa:whitenoise_sobolev_class} (see Exercise 10.11 in
\cite{schmuedgen2012unbounded_selfadjoint_operator_on_hilbert_space}).
%$L_{min}$ is too small to be self-adjoint. Indeed, by the Green's identity (only real functions for simplicity)
%\begin{align*}
%\int_{\mathcal{D}} (-\Delta u) v dx 
%= 
%\int_{\mathcal{D}}  u ((-\Delta)^\dagger v) dx  + \int_{\partial\mathcal{D}} u \partial_\nu v - (\partial_\nu u)  v d\sigma,
%\end{align*}
%where $(-\Delta)^\dagger$ is the formal adjoint and $\partial_\nu$ is the directional derivative in
%the direction of outward pointing normal $\nu$ to the surface element $d\sigma$. This implies that
%$D(-\Delta^\dagger)\supseteq W^{2,2}(\mathcal{D})\cap W_0^{1,2}(\mathcal{D})\supsetneq W_0^{2,2}(\mathcal{D})$.
Two possible domains on which $L$ is self-adjoint are 
(see Theorems~10.19 and 10.20 in \cite{schmuedgen2012unbounded_selfadjoint_operator_on_hilbert_space}):
\begin{align*}
\{f \in  W^{2,2}(\mathcal{D}): f|_{\partial\mathcal{D} } =0\}, \\
\{f \in  W^{2,2}(\mathcal{D}): \nabla f|_{\partial\mathcal{D} } =0\}.
\end{align*} 
These correspond to the \emph{Dirichlet} and \emph{Neumann} boundary conditions, respectively.
More sophisticated boundary conditions are possible as well,
see \cite{grubb2010distributions_and_operators,lions1972nonhomo_BVP_application_vol_1}.

In the Bayesian setup we model a function through a prior. When a true function is known to satisfy
certain boundary conditions, as in many problems involving differential forward operators, we 
can incorporate these in the prior by choosing an appropriate generating operator. For an operator $A$
defined in terms of the Laplacian and the same boundary conditions the smoothing condition 
\cref{eq:isomorphism_s=0} will be satisfied. The following is a another example of a pair
of $L$ and $A$.

%It is reasonable to consider Sobolev scales of functions that satisfy boundary conditions, since the
%existence and uniqueness of the solution of the forward problem \cref{eq:whitenoise_discussion_pde}
%is proved by establishing the fact that $L$ is isomorphism between Sobolev spaces satisfying
%boundary conditions, see \cite{lions1972nonhomo_BVP_application_vol_1}. As an immediate consequence,
%the isomorphism of the forward operator $A = L^{-1}$ is clear in the context of the corresponding inverse problem.

\begin{example}
	[Volterra]
	\label{ExampleTwoDimensionalVolterra}
	Consider the operator $A: L^2((0,1)^2)\to L^2((0,1)^2)$ on functions $f: (0,1)^2\to\RR$ 
	on the unit square satisfying the differential equation 
	$$D_{x,y} Af =f,\qquad\qquad D_{x,y}=\frac{\partial^2}{\partial x\partial y}.$$
	We can render the solution of the equation unique by imposing boundary conditions.
	Two solutions are given by 
	\begin{align*}
	Af(x,y)&=\int_0^x\!\!\int_0^y f(s,t)\,ds\,dt,\\
	A_0f(x,y)&=Af(x,y)-\int_0^1 \! Af(x,t)\,dt-\int_0^1 \! Af(s,y)\,ds+\int_0^1\!\!\int_0^1 \! Af(s,t)\,ds\,dt.
	\end{align*}
	The first satisfies the boundary conditions $Af(x,0)=Af(0,y)=0$,
	while the second is obtained from the first by subtracting its projection on the set
	of all functions of the form $(x,y)\mapsto g_1(x)+g_2(y)$, which forms the kernel of the
	differential operator. Other boundary conditions will still give different versions of the operator.
	
	We claim that $A_0$ is smoothing of order $\gamma=1$
	for the Hilbert scale generated by the root $L$ of $D_{x,y}^2$ with Dirichlet boundary condition, 
	while $A$ is smoothing relative to the scale of $L$ combined with Cauchy boundary condition.
	
	The scale under the Dirichlet boundary condition is generated by the orthogonal system
	of eigenfunctions $e_{k,l}: (x,y)\mapsto \sin(k\pi x)\sin(l\pi y)$, for $(k,l)\in\NN^2$, the tensor product
	of the basis of the one-dimensional Dirichlet-Laplacian as in Example~\ref{exa:hs_sobolev_laplacian}, with
	corresponding eigenvalues are $k^2l^2\pi^4$.
	By explicit calculation
	\begin{align*}
	Ae_{k,l}(x,y)&=\frac1{kl\pi^2}\bigl[\cos(k\pi x )\cos(l\pi y )-\cos(k\pi x )-\cos(l\pi y )+1\bigr],\\
	A_0e_{k,l}(x,y)&=\frac1{kl\pi^2}\cos(k\pi x )\cos(l\pi y ).
	\end{align*}
	The functions $(x,y)\mapsto \cos(k\pi x )\cos(l\pi y )$, for $(k,l)\in(\NN\cup\{0\})^2$ form an
	orthogonal basis of $L^2((0,1)^2)$. We conclude that for $f=\sum_{k,l}f_{k,l}e_{k,l}$,
	\begin{align*}
	\|Af\|^2&\simeq \sum_{k,l}\frac{f_{k,l}^2}{k^2l^2}+\sum_k\Bigl(\sum_{l}\frac{f_{k,l}}{kl}\Bigr)^2
	+\sum_l\Bigl(\sum_{k}\frac{f_{k,l}}{kl}\Bigr)^2+\Bigl(\sum_{k,l}\frac{f_{k,l}}{kl}\Bigr)^2,\\
	\|A_0f\|^2&\simeq \sum_{k,l}\frac{f_{k,l}^2}{k^2l^2}\simeq \|f\|_{-1}^2,
	\end{align*}
	where $\|\cdot\|_{-1}$ refers to the scale of $L$ with Dirichlet boundary condition.
	The first equation shows that the operator $A$ is not smoothing in this scale, but in general
	satisfies $\|Af\|\gtrsim \|f\|_{-1}$.
	
	On the other hand,  the Cauchy boundary condition generates the system of
	eigenfunctions $(x,y)\mapsto \cos ((k-1/2)\pi x)\cos ((l-1/2)\pi y)$, for $(k,l)\in\NN^2$. These
	can be seen to be also the eigenfunctions of $A^*A$, and hence the smoothing property of $A$ fits the`
	SVD framework, as in Example~\ref{exa:whitenoise_SVD}.
	
	The two versions $A$ and $A_0$ possess the same inverse operator, namely the differential operator $D_{x,y}$
	used for their definitions. This suggests that 
	from the point of view of reconstructing $f$ in the inverse problem it should not
	matter whether one is provided with a noisy version of either $Af$ or $A_0f$ as input data,
	seemingly contradicting the fact that the operators are smoothing in different scales. This paradox
	may be resolved by considering $A$ or $A_0$ as maps into the quotient space
	$L^2((0,1)^2)/N(D_{x,y})$, where $N$ denotes the kernel of the operator. The map $f\mapsto [Af]=[A_0f]$ into
	the class of $Af$ in this quotient space is injective and can be shown to 
	be appropriately smoothing (see \eqref{EqSmoothingConditionRelaxed}-\eqref{EqSmoothingConditionRelaxedtwo}), and consequently
	both scales can be used with both operators (cf.\ Remark~\ref{RemarkAAnul}).
\end{example}

\section{Gaussian Priors}
\label{sec:whitenoise_gaussian_prior}
If the function $f$ in \cref{eq:whitenoise_Y=Af+xi} is equipped with a Gaussian prior, then
the corresponding posterior distribution will be Gaussian as well. Furthermore, the posterior mean
will then be equal to the solution found by the method of  Tikhonov-type regularization
(see e.g.\, \cite{knapik2011bayesianmild,florens2016regularizing, stuart2010bayesianperspective}).
Although this allows to study the posterior mean and the full posterior distribution by direct methods,
in this section we derive the rate of posterior contraction from the general result
\Cref{thm:whitenoise_general_contraction_rate_subspace}. An advantage of this approach is
that the proof can be extended to mixtures of Gaussian priors, which is
important to obtain optimal recovery rates for true functions of different smoothness levels.
See \Cref{sec:whitenoise_gaussian_mixture_prior}. 

Centred Gaussian distributions on a separable Hilbert space correspond bijectively to covariance operators.
By definition a random variable $F$ with values in $H_0$ is Gaussian if
$\langle F, g\rangle_0$ is normally distributed, for every $g\in H_0$, and it has zero mean if these
variables have zero means. The variances of these variables can then be written as 
$$\E \langle F, g\rangle_0^2=\langle C g,g\rangle_0,$$
for a linear operator $C: H_0\to H_0$, called the \emph{covariance operator}. 
A covariance operator $C$ is necessarily self-adjoint, nonnegative, and
of \emph{trace class}, i.e., $\sum_{i\in\NN} \langle C\phi_i,\phi_i\rangle<\infty$, for some (and then every) orthonormal
basis $(\phi_i)_{i\in\NN}$ of $H_0$; and every operator with these properties generates a Gaussian distribution.

In the setting of a Hilbert scale $(H_s)_{s\in\RR}$ generated by the operator $L$ it is natural to choose a Gaussian prior 
with covariance operator of the form $L^{-2\a}$, for some $\a>0$. If $L^{-1}$ has eigenvalues
$\l_j$, then this operator is of trace class if $\sum_{j\in\NN}\l_j^{-2\a}<\infty$. Thus 
$\a$ must be chosen big enough for the Gaussian prior to exist as a `proper' prior on $H_0$.
For instance, if $\l_j\simeq j^{-1/d}$, then every choice $\a>d/2$ yields a proper prior.

This leads to the following theorem on posterior contraction rates for Gaussian priors, the proof of which is given in Section~\ref{SectionProofs}.

\begin{theorem}[Gaussian Prior]
	\label{thm:whitenoise_gaussian_prior}
	Consider a Hilbert scale $(H_s)_{s\in\RR}$ generated by an operator $L$ as in the preceding such that
	$L^{-1}: H_0\to H_0$ is compact with eigenvalues $\l_j$ satisfying $\l_j\simeq j^{-1/d}$.
	Suppose the operator $A: H_0\to G$ satisfies $\|Af\|\simeq\|f\|_{-\g}$, 
	assume that $f_0 \in H_\b$, for some $\b >0$, and let the prior be zero-mean Gaussian with 
	covariance operator $L^{-2\a}$,  for some $\a>d/2$.
	Then the posterior distribution satisfies, for sufficiently large $M>0$, 
	\begin{align*}
	\Pi_n\Bigl( f: \|f - f_0\|_0 > M n^{-((\a - d/2)\wedge\b)/(2\a+2\g)} \given Y^{(n)}\Bigr) \overset{\mathbb{P}^{(n)}_{f_0} }{\ra} 0.
	\end{align*}
\end{theorem}

If $F$ is distributed according to the prior in the preceding theorem, then $L^sF$ is also zero-mean
Gaussian distributed, with covariance operator $L^{2s-2\a}$, which has eigenvalues
$j^{-(2\a-2s)/d}$. For $s<\a-d/2$, this operator is
of trace class and hence $L^sF$ is a proper random variable in $H_0$. In other words,
the  distribution of $F$ gives probability 1 to $L^{-s}H_0=H_s$, for  every $s<\a-d/2$.
The prior in the preceding theorem can therefore be interpreted as being `almost' of regularity 
$\a-d/2$. The rate $n^{-((\a - d/2)\wedge\b)/(2\a+2\g)}$ is therefore comparable to
the rate obtained in Theorem~3.5 in \cite{ray2013nonconjugate} and Theorem~4.1 in
\cite{knapik2011bayesianmild} (with the scaling parameter fixed to 1), except that the parameter $\a$ 
in the latter references is denoted presently by $\a-d/2$. 

An improvement of the present theorem is that the covariance operator of the Gaussian prior
is not directly linked to the operator $A$, but only weakly so by \cref{eq:isomorphism_s=0}. 
For example, we may construct a prior by a random series (see Theorem I.23 in Appendix I.6,
\cite{vandervaart2017fundamentals}), in any basis corresponding to the smoothness scale.
We illustrate this below by using the wavelet basis for an inverse problem given by a
differential operator, after first noting that the singular value setup is covered as well.

\begin{example}
	[SVD]
	The scale of smoothness classes constructed in \Cref{ExampleSequenceSpaces}
	and \Cref{exa:whitenoise_SVD} is the Hilbert scale attached to
	the operator $L$ given by $Lf=\sum_{i\in\NN} b_if_i\phi_i$ defined on the domain of functions
	$f=\sum_{i\in \NN}f_i\phi_i$, with $\sum_{i\in \NN}b_i^2f_i^2<\infty$. Under assumption
	\eqref{EqSingularValuesArePowersOfb} this operator can also be expressed as $L=(A^*A)^{-1/(2\g)}$,
	and depends on the operator $A$ through its eigenfunctions.
	A Gaussian prior with covariance operator $L^{-2\a}$ corresponds to modelling
	the coefficients $f_i$ relative to the basis $\phi_i$ 
	as independent zero-mean normal variables $F_i$ with variances $b_i^{-2\a}$.
	This follows, because in that case $\E \langle F,g\rangle_0^2=\sum_{i\in \NN}b_i^{-2\a}g_i^2=\langle L^{-2\a}g,g\rangle_0^2$,
	for every $g\in H_0$.
	
	Thus in this case the prior coincides with the ones in the literature studied under the
	SVD framework, e.g. \cite{knapik2011bayesianmild,knapik2013bayesianextreme}. 
	In the present more general setting $L$ need not be directly linked to $A$,
	except that the operator must possess the smoothing property \Cref{cond:ill-posedness-mild}.
\end{example}

\begin{example}
	[Sobolev scales, wavelet prior]
	Let $\{\phi_{j,k}\}_{(j,k)\in \Lambda}$, be an $S$-regular orthonormal wavelet basis in
	$L^2(\mathbb{T})$, on $\mathbb{T}:=(0,2 \pi]$. Let $f_{j,k} = \int_\mathbb{T} f(x) \phi_{j,k}(x)\, dx$ be the wavelet coefficients
	of a function $f$. By Parseval's identity, the map $U: f\mapsto \{f_{j,k}\}$ is
	a unitary operator $U: L^2(\mathbb{T}) \to \ell^2(\Lambda)$. The multiplication operator $m:\{f_{j,k}\}\mapsto \{2^j f_{j,k}\}$ on
	$\ell^2(\Lambda)$ has $s$-th power given by  $m^s :\{f_{j,k}\}\mapsto \{2^{js} f_{j,k}\}$. Then 
	$L:= U^* m U$ has $s$-th power $L^s:= U^* m^s U$ and generates a Hilbert scale $(H_s)_{s\in\RR}$. 
	For $f\in H_s$, we have 
	\begin{align*}
	\|f\|_{H_s(\mathbb{T})}^2 = \sum_{j=0}^\infty 2^{2js} \sum_{k=0}^{2^j-1} f_{j,k}^2.
	\end{align*}
	This norm can be shown to be equivalent to the standard Sobolev norm, for $0 \leq s <S$.
	
	The Gaussian prior with covariance operator $L^{-2\a}$ can be represented by 
	a random series of the form 
	\begin{align*}
	F = \sum_{(j,k) \in \Lambda} F_{j,k} \phi_{j,k},
	\end{align*}
	where $F_{j,k}\sim \mathcal{N}(0, 2^{-2j\a})$ are independent random variables. 
	This prior corresponds to the Hilbert scale, but does not refer to an operator $A$.
	For instance, the eigenbasis of the operator in \Cref{exa:symm_eq} is the Fourier basis 
	(see \cite{kirsch2011introduction}), and not the wavelet basis. Thus 
	we have constructed a Gaussian prior that is not related to the eigenbasis,
	but attains the same contraction rate.
	
	It may be noted that the scale $(H_s)_{s\in\RR}$ is well defined for every $s\in \RR$,
	and with the preceding prior Theorem~\ref{thm:whitenoise_gaussian_prior} is applicable to 
	the full scale, and gives a contraction rate relative to the scale, which is optimal when
	$\b=\a-d/2$. However, the scale agrees with the Sobolev scale only for $\b<S$,
	and hence the optimality is in the Sobolev sense only if $\b<S$. This restriction is typical
	when working with an approximation scheme such as wavelets or splines. One can 
	of course choose a suitably large value of $S$, or may mix over multiple wavelet bases,
	as in the next section.
\end{example} 

As mentioned in \Cref{sec:whitenoise_Introduction}, there are many works on Bayesian inverse
problems with Gaussian priors. The setup of the preceding theorem is similar to
\cite{florens2016regularizing,agapiou_2013_gaussian_linear_inverse}, arguably closer to
\cite{agapiou_2013_gaussian_linear_inverse}. While we mainly treat the white noise case, our results
can be extended to cover the noise structure in \cite{agapiou_2013_gaussian_linear_inverse}, and hence also cover the
model in \cite{florens2016regularizing}. On the other hand, we differ from
\cite{agapiou_2013_gaussian_linear_inverse} in the following sense. First, unlike Assumption 3.1 in
\cite{agapiou_2013_gaussian_linear_inverse}, our characterization of the smoothing property of the
operator $A$, i.e. \Cref{cond:ill-posedness-mild}, is simple, and in principle, our setup can also be
extended to severely ill-posed problems, see \Cref{sec:whitenoise_extension}.
%which cannot be done using the techniques in \cite{agapiou_2013_gaussian_linear_inverse}. 
Second, our proof strategy is different, as we do not use Gaussian conjugacy, which is the main tool in
\cite{agapiou_2013_gaussian_linear_inverse}. This also allows us to obtain posterior
contraction rates for non-conjugate priors in \Cref{sec:whitenoise_random_series_prior},
and for Gaussian mixtures in \Cref{sec:whitenoise_gaussian_mixture_prior}.

\section{Gaussian Mixtures}
\label{sec:whitenoise_gaussian_mixture_prior}

The posterior contraction rate resulting from a zero-mean Gaussian prior with covariance operator $L^{-2\a}$,
as considered in \Cref{sec:whitenoise_gaussian_prior},  is equal to the minimax rate $n^{-\b/(2\b + 2\g +d)}$ 
(see \cite{cohen2004adaptive_galerkin}) only when
$\a - d/2 = \b$, i.e., when the prior smoothness $\a-d/2$ matches the true smoothness $\b$. 
By mixing over Gaussian priors of varying smoothness the minimax rate can often be
obtained simultaneously for a range of values $\b$ (cf.\ \cite{KnapikPTRF}, \cite{vandervaartICM},
\cite{szabo2013empirical}). In this section we consider 
mixtures of the mean-zero Gaussian priors with covariance operators $\t^2L^{-2\a}$ over 
the `hyperparameter' $\t$. Thus the prior $\Pi$ is the distribution of $\t F$, where $F$ is a zero-mean
Gaussian variable in $H_0$ with covariance operator $L^{-2\a}$,
as in \Cref{sec:whitenoise_gaussian_prior}, and $\t$ is an independent scale parameter.
The variable $1/\t^a$ may be taken to possess a Gamma distribution for some given $0<a\le 2$,
or, more generally, should satisfy the following mild condition.

\begin{condition}
	\label{cond:whitenoise_gaussian_mixture}
	The distribution $Q$ of $\t$ has support $[0,\infty)$ and satisfies 
	$$\begin{cases}
	-\log Q\bigl((t,2t)\bigr)\lesssim t^{-2}, & \text{ as } t\da 0,\\
	-\log Q\bigl((t,2t)\bigr)\lesssim t^{d/(\a-d/2)}, & \text{ as } t\ra \infty.
	\end{cases}$$
\end{condition}

\begin{theorem}
	[Gaussian mixture prior]
	\label{thm:whitenoise_gaussian_mixture_prior}
	Consider a Hilbert scale $(H_s)_{s\in\RR}$ generated by an operator $L$ as in the preceding such that
	$L^{-1}: H_0\to H_0$ is compact with eigenvalues $\l_j$ satisfying $\l_j\simeq j^{-1/d}$.
	Suppose the operator $A: H_0\to G$ satisfies $\|Af\|\simeq\|f\|_{-\g}$, 
	assume that $f_0 \in H_\b$, for some $\b \in (0,\a]$, and let the prior be a mixture 
	of the zero-mean Gaussian distributions with covariance operators $\t^2 L^{-2\a}$ 
	over the parameters $\t$ equipped with a prior satisfying \Cref{cond:whitenoise_gaussian_mixture}, for some $\a>d/2$.
	Then the posterior distribution satisfies, for sufficiently large $M>0$, 
	\begin{align*}
	\Pi_n\Bigl( f: \|f - f_0\|_0 > M n^{-\b/(2\b + 2\g + d)} \given Y^{(n)}\Bigr) \overset{\mathbb{P}^{(n)}_{f_0} }{\ra} 0.
	\end{align*}
\end{theorem}

\noindent
The proof is given in Section~\ref{SectionProofs}.

%%%%%%%%%%%%%%%%%%%%%%%%%%%%%%%%%%%%%%%%
\section{Discussion and Comments}
\label{sec:whitenoise_extension}

In this section we comment on the present setup and discuss directions in which the results in this article can be extended.

\subsubsection*{Coloured Noise}
We have examined the case that the noise $\xi$ in model \cref{eq:whitenoise_Y=Af+xi} is white
noise. Statistical estimation in the case that the noise is a proper centred Gaussian random element 
in $G$,  as studied in \cite{florens2016regularizing},  is easier in terms of minimax rates 
(if in both cases the noise is scaled to the same unit), 
as this would imply that the noise is less variable. By inspection of our proofs one sees that the
concentration inequalities that drive the testing criterion remain valid if the covariance operator of
the noise is bounded above by the identity, as is assumed in 
\cite{bissantz2007inverseregularization,agapiou_2013_gaussian_linear_inverse}. As a consequence, the proof of 
\Cref{thm:whitenoise_general_contraction_rate_subspace} goes through and the theorem remains valid,
as do the corollaries in the later sections. However, for truly coloured noise the result may be suboptimal, as one may expect
a faster posterior contraction rate, which will incorporate the decrease of the noise variance in certain directions.
The methods of the present paper can be adapted to this case 
as long as the covariance operator fits the scale of smoothness classes,  as in \cite{florens2016regularizing}. 
A sharp result in full generality may be difficult to attain, as it will be the outcome of the interaction of the 
directions of decrease in the noise, the true parameter and the prior.

\subsubsection*{Approximation Numbers of Embeddings}

In the corollaries to the main result we have assumed that the approximation numbers $\d(j,s)$ of
the canonical embedding $\iota: H_s \to H_0$ are of polynomial order $j^{-s/d}$. This order matches
the approximation numbers of Sobolev spaces on $d$-dimensional, bounded domains, and seems common.
Other decay rates do arise, e.g., an exponential rate in severely ill-posed problems 
(as in the heat equation considered in \cite{knapik2013bayesianextreme}),
or a  logarithmic rate (as in \cite{castillo2013bvmwhitenoise}). The general \Cref{thm:whitenoise_general_contraction_rate_subspace}
remains valid, but its corollaries must be adapted. For Gaussian priors in logarithmic or exponential scales,
this is relatively straightforward using the general theory of approximation numbers, which relates
these to singular values and metric entropy. See the discussion in \Cref{sec:whitenoise_entropy}.

%%%%%%%%%%%%%%%%%%%%%%%%%%%%%%%%%%%%%%%%
\section{Proofs}
\label{SectionProofs}

\begin{lemma}
	\label{LemmaKL}
	For $\q=(\q_1,\q_2,\ldots)$ let $P_\q$ be the distribution of the random element $(X_1+\q_1,X_2+\q_2,\ldots)$
	in $\RR^\infty$ for $X_1,X_2\ldots$ i.i.d.\ mean-zero normal variables with variance $\s^2$. If $\q\in\ell^2$,
	then $P_\q$ is absolutely continuous relative to $P_0$ with log likelihood
	$$\log \frac{dP_\q}{dP_0}(X_1,X_2,\ldots)=\frac{1}{\s^2}\sum_{i=1}^\infty \q_iX_i-\frac{1}{2\s^2} \sum_{i=1}^\infty\q_i^2$$
	where the first series converges almost surely and in second mean. The expectation and variance
	of minus this variable are $\sum_{i=1}^\infty \q_i^2/(2\s^2)$ and twice this quantity, respectively.
\end{lemma}

\begin{proof}
	That the series converges in $L^2$ is clear from the fact that $\q\in\ell^2$; the almost sure  convergence next follows from the It\^o-Nisio theorem.
	The expectation and variance of the right side are  easy to compute as limits.
	
	Write $\Lambda_\infty$ for the right side of the display, and $\Lambda_n$ for the expression obtained by replacing the infinite
	sums by the sums from 1 to $n$. Thus $\Lambda_n\ra \Lambda_\infty$ almost surely. Since
	$\E_0 e^{2\Lambda_n}= e^{\sum_{i=1}^n\q_i^2/\s^2}$ is uniformly bounded in $n$, it follows that $e^ {\Lambda_n}$
	is uniformly integrable and hence converges in mean to $e^{\Lambda_\infty}$. In particular, the mean of the latter
	variable is 1, the mean of the former variables.
	
	It follows that the Borel measure on $\RR^\infty$  defined by $B\mapsto \E_0 1_B(X)e^{L_\infty}$ is a probability measure.
	For every Borel set $B$ it is the limit of $\E_0 1_B(X)e^{L_n}$, which is $P_\q(B)$ if $B$ depends only on the
	first $n$ coordinates, as $e^{L_n}$ is the density of the distribution of $(X_1+\q_1,\ldots, X_n+\q_n)$ with
	respect to its distribution at $\q=0$.  Since the Borel $\s$-field on $\RR^\infty$ is generated 
	by the algebra of all cylinder sets, it follows that $P_\q$ and the measure $B\mapsto \E_0 1_B(X)e^{L_\infty}$ agree.
\end{proof}

\subsection{Proof of \texorpdfstring{\Cref{thm:whitenoise_random_series_prior_mild}}{Random Series Priors} }
The theorem is a corollary to \Cref{thm:whitenoise_general_contraction_rate_subspace}
and uses arguments as in the proof of Proposition~3.2 in \cite{ray2013nonconjugate}. 

First we determine $\e_n$ to satisfy the prior mass condition \cref{EqPriorMass} of the
direct problem. Let $P_j$ be the projection onto the linear span of the first $j-1$ basis elements
$\phi_i$. By the assumption on $A$ and the triangle inequality, for any $i_n\in\NN$, 
\begin{align}
\|Af-Af_0\|&\lesssim \|f - f_0\|_{-\g} \lesssim \|f - P_{i_n}f_0\|_{-\g} + \|P_{i_n}f_0 - f_0\|_{-\g}\nonumber\\
&\lesssim  \|f - P_{i_n}f_0\|_{-\g} + \d(i_n,\g)\d(i_n,\b)\|f_0\|_\b,
\label{EqApproximateAf}
\end{align}
by  \cref{eq:whitenoise_projection_estimate_-}, if $0\le\b,\g< S$. 
Here $\d(i_n,\g)\d(i_n,\b)=i_n^{-(\g+\b)/d}\simeq \e_n$ if $i_n\simeq \e_n^{-d/(\g+\b)}$. 

By the orthogonality of the basis $(\phi_i)$, the function $\phi_j$ is orthogonal to the
space $V_j$ spanned by $(\phi_i)_{i<j}$. Hence $P_j\phi_j=0$, so that
$\|\phi_j\|_{-\g}\le \d(j,\g)\|\phi_j\|_0\lesssim j^{-\g/d}$, for every $j$,
by \cref{eq:whitenoise_projection_estimate_-}.
Consequently, for $f = \sum_{i=1}^{i_n-1}f_i \phi_i \in V_{i_n}$ and $f_0=\sum_if_{0,i}\phi_i$,
by the triangle inequality,
\begin{align*}
\|f -P_{i_n} f_0\|_{-\g} \lesssim \sum_{i=1}^{i_n-1}|f_i - f_{0,i}| i^{-\g/d}.
\end{align*}
It follows that there exists a constant $a>0$ such that
\begin{align*}
&\Pi \bigl(f: \|f -P_{i_n} f_0\|_{-\g} < a\e\bigr) 
\geq \Pi \Bigl(\bigl((f_i),M\bigr): \sum_{i=1}^{i_n-1}|f_i - f_{0,i}| i^{-\g/d} < \e, M=i_n-1\Bigr)\\
%&\ge \Pi \bigl( |f_i - f_{0,i}| i^{-\g/d} \leq \frac{\e_n}{i_n -1},\ i=1,\cdots, i_n \bigr)\\
&\qquad\ge \prod_{i=1}^{i_n}\Pi \left(f_i:  |f_i - f_{0,i}|  <\frac{\e\, i^{\g/d}}{i_n}\right) \Pi(M=i_n-1)\\
&\qquad\ge \prod_{i=1}^{i_n}\int_0^{\e\, i^{\g/d}/(\k_i i_n)}\! p\Bigl(x+\frac{f_{0,i}}{\k_i}\Bigr)\,dx\   e^{-b_1i_n},
\end{align*}
in view of  \Cref{cond:whitenoise_random_series_prior}. By \cref{eq:whitenoise_random_series_prior_tail_weight}
of the latter assumption, the integral $\int_0^rp(x+\mu)\,dx$ is bounded below by a constant times  $r e^{-C(r+|\mu|)^w}$.
It follows that for $\e$ such that $\e\, i^{\g/d}/(\k_i i_n)\le 1$, for $i\le i_n$,
the preceding display is lower bounded by a multiple of 
$$\e^{i_n} \Bigl[\prod_{i=1}^{i_n}\frac {i^{\g/d}}{\k_i i_n}\Bigr] 
\exp\Bigl[-C\sum_{i=1}^{i_n}\Bigl(1+\frac{|f_{0,i}|}{\k_i}\Bigr)^w \Bigr]\,  e^{-b_1i_n}.$$
By \cref{EqBoundsScalingKappa}, we have $i^{\g/d}/\k_i\gtrsim (1/i)^{\g/d-\a}$, which is bounded
below by 1 if $\g/d-\a\ge0$ and by $(1/i_n)^{\a-\g/d}$ otherwise, and hence always by $(1/i_n)^\a$. This shows that
the first term in square brackets is bounded below by $(a_2/i_n^{\a+1})^{i_n}$, for some $a_2>0$. 
Since $f_0\in H_\b$, by assumption, the norm duality \cref{EqNormDuality} 
gives that $|f_{0,i}| = |\langle f_0,\phi_i \rangle_0|\leq \|f_0\|_\b \|\phi_i\|_{-\b} \lesssim i^{-\b/d}$.
Together with \cref{EqBoundsScalingKappa} this
gives that $|f_{0,i}|/\k_i\lesssim i^{(\b_0-\b)/d}(\log i)^{1/w}\le (\log i)^{1/w}$, whence minus the exponent
in the second term in square brackets is bounded by a multiple of $i_n\bigl(1+(\log i_n)^{1/w}\bigr)^w$.
We conclude that there exists a constant $a_3>0$ such that 
$$\Pi \bigl(f: \|f -P_{i_n} f_0\|_{-\g} < a \e\bigr)\ge \e^{i_n}e^{-a_3i_n\log i_n}  e^{-b_1 i_n},$$
for every $\e>0$ such that $\e\, i^{\g/d}/(\k_i i_n)\le 1$, for every $i\le i_n$.
Since $i^{\g/d}/\k_i\lesssim i^{(\g+\b_0)/d}(\log i)^{1/w}$,  again by \cref{EqBoundsScalingKappa}, 
a sufficient condition for the latter is that $\e\, i_n^{(\g+\b_0)/d}(\log i_n)/i_n\le 1$.

Combining this with \cref{EqApproximateAf}, we see that  \cref{EqPriorMass} is satisfied for
$\e_n$ such that there exists $i_n$ with 
$$i_n^{-(\g+\b)/d}\lesssim \e_n, \qquad i_n\log i_n\lesssim n\e_n^2, \qquad  \e_n i_n^{(\g+\b_0)/d}(\log i_n)\le i_n.$$
This leads to the rates
$$\e_n \simeq (\log n/n)^{(\b + \g)/(2 \b + 2 \g +d)}, \qquad 
i_n\simeq (n/\log n)^{d/(2\b+2\g+d)}.$$ 
(The third requirement is easily satisfied and remains inactive.) We can choose a sufficiently
large proportionality constant in $\simeq$ when defining $\e_n$, so that \cref{EqPriorMass} is satisfied
for $\e_n$, since the left and right sides of \cref{EqPriorMass} are increasing 
and decreasing in $\e_n$, respectively.

Since the Galerkin projection $f^{(j)}$ is equal to $f$ itself if $f\in V_j$, we have that
$\|f^{(j_n)}-f\|_0=0$ for the random series $f=\sum_{i=1}^Mf_i\phi_i$ if $M< j_n$. By (ii) of 
\Cref{cond:whitenoise_random_series_prior} it follows that, for some $b_2'>0$ and every $\eta_n>0$,
$$\Pi\bigl(f: \|f^{(j_n)}-f\|_0>\eta_n\bigr)\le \Pi\bigl(M\ge j_n\bigr)\le e^{-b_2' j_n}.$$
Hence \cref{EqPriorApproximation} is satisfied for $j_n =n \e_n^2/(4b_2')$. 
Thus we choose 
$$j_n\simeq n^{d/(2\b+2\g+d)}(\log n)^{(2\b+2\g)/(2\b+2\g+d)},$$ 
with a sufficiently large constant in $\simeq$.
Then \cref{EqjEpsilon} is satisfied and it remains to solve $\eta_n$ from  \cref{EqEtaEpsilonDelta} and \cref{EqEtaDelta}. 
This leads to the inequalities 
\begin{align*}
\eta_n &\ge  \e_n j_n^{\g/d}\simeq n^{-\b /(2\b + 2 \g + d)} (\log n)^{(1+2\g/d )(\b + \g )/(2\b + 2 \g + d)},\\
\eta_n& \ge j_n^{-\b/d}\simeq n^{-\b/(2\b+2\g+d)}(\log n)^{-\b (2\b+2\g)/((2\b+2\g+d)d)}.
\end{align*}
The rate is the maximum of the rates at the right hand sides, which coincides with the first
rate. This concludes the proof.

%%%%%%%%%%%%%%%%%%%%%%%%%%%%%%%%%%%%%%

\subsection{Proof of \texorpdfstring{\Cref{thm:whitenoise_gaussian_prior}}{Gaussian Priors} }
\label{subsec:whitenoise_proof_gaussian_prior}
The theorem is a corollary to \Cref{thm:whitenoise_general_contraction_rate_subspace}.
The main tasks are to determine $\e_n$ satisfying the prior mass condition \cref{EqPriorMass} of the
direct problem, and next to identify $\eta_n$ from the prior mass condition
\cref{EqPriorApproximation} and the other conditions.

The first task is achieved in the following lemma.

\begin{lemma}
	\label{lem:whitenoise_gaussian_prior_concentration}
	Under the assumptions of \Cref{thm:whitenoise_gaussian_prior}, for $f_0\in H_\b$,
	as $\e\da 0$,
	\begin{equation}
	\label{eq:log:prob}
	\begin{split}
	-\log\Pi\bigl(f: \|Af-Af_0\|<\e\bigr) \lesssim
	\begin{cases}
	\e^{-d/(\a + \g -d/2)},&\text{ if } d/2 <\a \leq \b + d/2,   \\
	\e^{-(2\a - 2 \b)/(\b + \g)},&\text{ if } \a >\b + d/2.
	\end{cases}
	\end{split}
	\end{equation}
\end{lemma}

\begin{proof}
	Since by assumption $\|Af-Af_0\|\simeq \|f-f_0\|_{-\g}$, the probability  in the left side
	is the decentered small ball probability  $\Pi\bigl(f: \|f-f_0\|_{-\g}<a\e\bigr)$ of the Gaussian random variable
	$F$ distributed according to the prior and viewed as map into $H_{-\g}\supset H_0$, for some $a>0$. 
	Because $F$ has covariance operator $L^{-2\a}$ as a map in $H_0$, its reproducing kernel Hilbert space 
	(or Cameron-Martin space) $\HH$
	(which does not depend on its range space) is equal to the range of $L^{-\a}$ under the norm $\|L^{-\a}h\|_\HH=\|h\|_0$
	(see e.g., Example~I.14 of \cite{vandervaart2017fundamentals}).
	Since $L^{-\a}: H_0\to H_\a$ is a norm isometry, by (iii) of \Cref{prop:whitenoise_hilbert_scale_properties},
	this is the Hilbert space $H_\a$ with its natural norm $\|\cdot\|_\a$.
	The left side of \eqref{eq:log:prob} is therefore up to constants equivalent to 
	\begin{align}
	\label{EqSmallBallGeneral}
	\inf_{h\in H_\a: \|h - f_0\|_{-\g}< \e } \|h\|_\a^2 - \log \Pi\bigl(\|f\|_{-\g}<\e\bigr).
	\end{align}
	See \cite{KuelbsLi,KuelbsLiLinde,vandervaart2008gaussianprior}, or Section~11.2, in particular, Proposition~11.19 in 
	\cite{vandervaart2017fundamentals}.
	
	By \cref{eq:whitenoise_projection_estimate_-} $\|P_jf_0-f_0\|_{-\g}\lesssim \d(j,\g)\d(j,\b)\|f_0\|_\b$, 
	which is bounded above by $\e$ for $j \simeq \e^{-d/(\b + \g)}$. Thus for this value of $j$ the first term 
	in \cref{EqSmallBallGeneral} is bounded above by
	\begin{align*}
	\|P_j  f_0\|_\a\lesssim& 
	\begin{cases}
	\|P_j f_0\|_\b ,&\text{ if } \a\leq\b,\\
	1/\d(j, \a - \b)\|P_jf_0\|_\b,&\text{ if } \a > \b
	\end{cases}
	\end{align*}
	by \cref{EqStabilityPropertyExtended}. Here $\|P_j  f_0\|_\b\le \|P_jf_0-f_0\|_\b+\|f_0\|_\b\le \bigl(\d(j,0)+1\bigr)\|f_0\|_\b$,
	by \cref{EqApproximationPropertyExtended}.
	It follows that the contribution of the decentering in \cref{EqSmallBallGeneral} is of order 1 if $\a\le\b$ and is bounded above
	by a term of order $\e^{-2(\a - \b)/(\b + \g)}$ if $\a>\b$.
	
	By \Cref{lem:whitenoise_metric_entropy_hilbert_scales}, the metric entropy 
	$\log N\bigl(\e, \{f \in H_\a: \|f\|_\a \leq 1\}, \|\cdot\|_{-\g}\bigr)$ is of the order  $\e^{-d/(\a + \g)}$. 
	Hence, by \cite{KuelbsLi} (see Lemma 6.2 in \cite{vandervaart2008RKHS}),
	\begin{align*}
	-\log \Pi\bigl(\|f\|_{-\g}< \e\bigr) \simeq \e^{-d/(\a +\g - d/2)}.
	\end{align*}
	Finally, the assertion of the lemma follows from discussion by cases.
\end{proof}

It follows that \cref{EqPriorMass} is satisfied for
\begin{align}
\label{eq:whitenoise_gaussian_prior_rate_from_small_ball_probability}
\e_n \geq n^{-(\b \wedge (\a - d/2) + \g)/( 2\a + 2 \g)}.
\end{align}
The next step of the proof is to bound the prior probability in \cref{EqPriorApproximation}.

\begin{lemma}
	\label{lem:whitenoise_gaussian_prior_approximation}
	Under the assumptions of \Cref{thm:whitenoise_gaussian_prior},
	there exist $a,b>0$, such that for every $j\in \NN$ and $t>0$,
	$$\Pi\bigl(f: \|f^{(j)}-f\|_0>t+ a j^{1/2-\a/d}\bigr)\le e^{-bt^2j^{2\a/d}}.$$
\end{lemma}

\begin{proof}
	We have $f^{(j)}-f=(R_jA-I)f$, for $R_j=A^{-1}Q_j$. Therefore, the probability
	on the left concerns the random variable $(R_jA-I)F$, if $F$ is a variable distributed
	according to the prior $\Pi$. Since $F$ is zero-mean normal with
	covariance operator $L^{-2\a}$, this variable is zero-mean Gaussian with 
	covariance operator $(R_jA-I) L^{-2\a}(R_jA-I)^*$. We shall compute the weak and strong
	second moments of the variable $(R_jA-I)F$, and next apply Borell's
	inequality for the norm of a Gaussian variable to obtain the exponential bound.
	
	Because $\langle (R_jA-I)F,g\rangle_0=\langle F, (R_jA-I)^*g\rangle_0$ is zero-mean Gaussian 
	with variance $\|L^{-\a}(R_jA-I)^*g\|_0^2=\|(R_jA-I)^*g\|_{-\a}^2$, 
	the weak second moment of $(R_jA-I)F$ is given by
	$$\sup_{\|g\|_0\le 1}\E \langle (R_jA-I)F,g\rangle_0^2=\sup_{\|g\|_0\le 1}\|(R_jA-I)^*g\|_{-\a}^2.$$
	By the norm duality \cref{EqNormDuality}, the right side is equal to
	$$\sup_{\|g\|_0\le 1}\sup_{\|f\|_\a\le 1}\langle f, (R_jA-I)^*g\rangle_0^2
	\le \sup_{\|f\|_\a\le 1}\|(R_jA-I)f\|_0^2\lesssim \d(j,\a)^2.$$
	in view of \cref{eq:whitenoise_galerkin_solution_error_estimate}.
	
	The strong second moment of the Gaussian variable $(R_jA-I)F$
	is equal to the trace of its covariance operator. 
	As $\trace(S^*S)=\sum_i\|S \phi_i\|^2=\sum_i\sum_j\langle S\phi_i, \phi_j\rangle^2=\sum_i\|S^*\phi_i\|^2$,
	for any orthonormal basis $(\phi_i)$ and operator $S$, we have 
	$$\E \|(R_jA-I)F\|_0^2= \sum_{i\in\NN}\|(R_jA-I)L^{-\a}\phi_i\|_0^2.$$
	For the orthonormal basis of eigenfunctions of $L^{-1}$ and $V_j$ the span of the first
	$j-1$ of these eigenfunctions, as in \Cref{PropositionEigenvaluesApproximationNumbers},
	$L^{-\a}V_j\subset V_j$, and hence $(R_jA-I)L^{-\a}\phi_i$ vanishes for $i<j$. 
	For $i\ge j$ the latter element is the difference $g^{(j)}-g$ of the Galerkin solution $g^{(j)}$ to $g=L^{-\a}\phi_i$.
	Therefore, by  \cref{eq:whitenoise_galerkin_solution_error_estimate} the preceding display is bounded above by
	a multiple of 
	$$%\sum_{i\ge j}\|(R_jA-I)L^{-\a}\phi_i\|_0^2\lesssim 
	\sum_{i\ge j}\d(i,\a)^2\|L^{-\a}\phi_i\|_\a^2=\sum_{i\ge j}\d(i,\a)^2\|\phi_i\|_0^2\lesssim j^{1-2\a/d},$$
	where we used the estimate $\sum_{i>j} i^{-b} \leq j^{1 -b}/(b -1)$, for $b>1$.
	
	Since the first moment of $\|(R_jA-I)F\|_0$ is bounded by the root of its
	second moment, the lemma follows by Borell's inequality
	(see e.g. Lemma~3.1 and subsequent discussion in \cite{LedouxTalagrand}).
\end{proof}

For $t^2= 4 n\e_n^2/(bj_n^{2\a/d})$ and $j=j_n$ the bound in the preceding lemma becomes $e^{-4n\e_n^2}$.
Hence \cref{EqPriorApproximation} is satisfied for
$$\eta_n\gtrsim \sqrt{n} \e_n j_n^{-\a/d}+ j_n^{1/2-\a/d}.$$
Here we choose $\e_n$ the minimal solution that satisfies the direct prior mass condition \cref{EqPriorMass},
given in \cref{eq:whitenoise_gaussian_prior_rate_from_small_ball_probability}.
Next we solve for $\eta_n$ under the constraints \cref{EqEtaEpsilonDelta} and \cref{EqEtaDelta}.
The first of these constraints,  $j_n\le n\e_n^2$, shows that the first term on the right side of the preceding display  always
dominates the second term. Therefore, we obtain the requirements
$j_n\le n\e_n^2$ and 
\begin{align*}
\eta_n&\ge \sqrt n\, n^{-(\b \wedge (\a - d/2) + \g)/( 2\a + 2 \g)} j_n^{-\a/d},\\
\eta_n&\ge n^{-(\b \wedge (\a - d/2) + \g)/( 2\a + 2 \g)}j_n^{\g/d},\\
\eta_n&\ge j_n^{-\b/d}.
\end{align*}
Depending on the relation between $\a$ and $\b + d/2$, two situations need to be discussed separately.

\begin{enumerate}[(i)]
	\item $\a \le \b + d/2$. 
	We choose $j_n \simeq n^{d/(2\a +2\g) } = n\e_n^2$ and then see that
	the first two requirements in the preceding display both reduce
	to $\eta_n \ge  n^{-(\a - d/2)/(2\a + 2\g)}$, while the third becomes
	$\eta_n\ge n^{-\b/(2\a+2\g)}$ and becomes inactive. 
	\item $\a > \b + d/2$. We choose $j_n \simeq n^{d/(2\a+2\g)}\le n\e_n^2$, 
	and then see that all three requirements reduce to $\eta_n\ge n^{-\b/(2\a+2\g)}$.
\end{enumerate}
Finally, we apply \Cref{thm:whitenoise_general_contraction_rate_subspace} to complete the proof.

\subsection{Proof of \texorpdfstring{\Cref{thm:whitenoise_gaussian_mixture_prior}}{Gaussian Mixtures}}
Let $\Pi_\t$ denote the zero-mean Gaussian distribution on $H$ with covariance
operator $\t^2L^{-2\a}$ (where $\a>d/2$).

\begin{lemma}
	\label{lem:whitenoise_gaussian_mixture_prior_concentration}
	Under the assumptions of \Cref{thm:whitenoise_gaussian_mixture_prior}, for $f_0\in H_\b$ and $\b\le \a$,
	as $\e\da 0$,
	\begin{align*}
	-\log\Pi_\t \bigl(f: \|Af-Af_0\|<\e\bigr) \lesssim
	\frac1{\t^{2}} \left(\frac1\e\right)^{(2\a - 2\b)/(\b + \g)}
	+ \left(\frac{\t}{\e}\right)^{d/(\a + \g -d/2)}.
	\end{align*}
\end{lemma}

\begin{lemma}
	\label {LemmaPriorConcentrationRough}
	Under the assumptions of \Cref{thm:whitenoise_gaussian_mixture_prior}, for $f_0\in H_\b$  and $\b\le \a$,
	as $\e\da 0$,
	\begin{align*}
	-\log\Pi_\t \bigl(f: \|f\|_0<\e\bigr) \gtrsim \left(\frac{\t}{\e}\right)^{d/(\a  -d/2)}.
	\end{align*}
\end{lemma}

\begin{lemma}
	\label{lem:whitenoise_gaussian_mixture_prior_approximation}
	Under the assumptions of \Cref{thm:whitenoise_gaussian_mixture_prior},
	there exist $a,b>0$ such that, for every $j\in\NN$ and $x, \t>0$,
	$$\Pi_\t\bigl(f: \|f^{(j)}-f\|_0> \t x + \t a j^{1/2-\a/d}\bigr)\le e^{- bx^2 j^{2\a/d}} $$
\end{lemma}

\begin{proof}[Proofs]
	The proof of the first lemma follows the same lines as the proof of \Cref{lem:whitenoise_gaussian_prior_concentration},
	except that now the Cameron-Martin space of the measure $\Pi_\t$ on $H_{-\g}$ is $H_\a$ 
	equipped with the norm $\|\cdot\|_\HH = \frac{1}{\t} \|\cdot\|_\a$ rather than its natural norm.
	The second lemma follows similarly, but considers the centered probability only.
	The third  lemma is immediate from 
	\Cref{lem:whitenoise_gaussian_prior_approximation} as $\Pi_\t$ is the law of $\t F$, for
	$F$ the Gaussian variable with the law $\Pi$ as in the latter lemma, and the map
	$f\mapsto f^{(j)}-f$ is linear.
\end{proof}

As preparation for the proof of  \Cref{thm:whitenoise_gaussian_mixture_prior}, we first show
that the minimax rate can be obtained by a Gaussian prior with the deterministic scaling, dependent on $\b$,
given by
\begin{equation}
\label{EqDefTaun}
\t_n = n^{(\a - d/2 - \b)/(2\b + 2 \g +d)}.
\end{equation}

\begin{theorem}
	\label{lem:whitenoise_gaussian_mixture_deterministic_scaling}
	Assume the conditions on the Hilbert scale, the forward operator $A$ and the true parameter $f_0$ in
	\Cref{thm:whitenoise_gaussian_prior} hold.  Suppose that the priors $\Pi$ are zero-mean Gaussian
	with covariance operators $\t_n^2 L^{-2 \a}$ with $\t_n$ as given in \cref{EqDefTaun} and $\a >d/2$.
	Then for $\b\le \a$, the posterior distribution satisfies, for
	sufficiently large $M>0$,
	\begin{align*}
	%\label{eq:whitenoise_contraction_rate_gaussian_mixture_deteministic_scaling}
	\Pi_n\Bigl( f: \|f - f_0\|_0 > M n^{-\b/(2\b+2\g+d)} \given Y^{(n)}\Bigr) \overset{\mathbb{P}^{(n)}_{f_0} }{\ra} 0.
	\end{align*}
\end{theorem}

\begin{proof}
	The theorem is a corollary to \Cref{thm:whitenoise_general_contraction_rate_subspace}.
	The proof follows the same lines as the proof of \Cref{thm:whitenoise_gaussian_prior}.
	By  \Cref{lem:whitenoise_gaussian_mixture_prior_concentration}, inequality \cref{EqPriorMass} is satisfied for
	\begin{align*}
	\e_n\gtrsim n^{ -(\b+\g)/(2\b + 2\g + d)}.
	\end{align*}
	By \Cref{lem:whitenoise_gaussian_mixture_prior_approximation}, inequality \cref{EqPriorApproximation} is satisfied for
	\begin{align*}
	\eta_n \gtrsim \t_n \big(\sqrt{n}\e_n j_n^{-\a/d} + j_n^{1/2 -\a/d}\big).
	\end{align*}
	We choose $j_n \simeq n \e_n^2$, and the minimal solution $\e_n = n^{ -(\b+\g)/(2\b + 2\g + d)}$ to the second
	last display. It is then straightforward to verify that \cref{EqEtaEpsilonDelta}, \cref{EqEtaDelta} 
	and \cref{EqPriorApproximation} are satisfied for $\eta_n \simeq n^{-\b/(2\b +2 \g + d)}$.
\end{proof}

\Cref{thm:whitenoise_gaussian_mixture_prior} is a corollary of
\Cref{TheoremGeneralAdaptation}, with the choices 
\begin{alignat*}{3}\eta_n &\simeq n^{-\b/(2\b +2\g + d)}, \qquad\qquad& \e_n &\simeq n^{-(\b + \g)/(2\b + 2\g + d)},\\
j_n&\simeq n\e_n^2 = n^{d/(2\b + 2\g +d)}.&&
\end{alignat*}
Conditions \cref{EqjEpsilon}, \cref{EqEtaEpsilonDelta}, and \cref{EqEtaDelta} are satisfied
for these choices. It remains to 
verify \cref{EqPriorMass}, and \cref{EqPriorApproximationExtended}--\cref{EqPriorMassRough}.

For ease of notation, for the moment, define $\eta_n$ and $\e_n$ as in the preceding display, with exact equality
(i.e., with the constant  set equal 1). Let $\t_n$ be the `optimal' scaling rate defined in \cref{EqDefTaun}.

Verification of \cref{EqPriorMass}. 
For $\t\simeq\t_n$ and $\e\simeq\e_n$ as given and $\b\le \a$,
both terms in the right side of \Cref{lem:whitenoise_gaussian_mixture_prior_concentration}
are of the  order $n\e_n^2$. The lemma yields, for $\t_n\le\t\le 2\t_n$ and some constant $a_1>0$,
$$-\log \Pi_{\t}\bigl(f: \|Af - Af_0\|< \e_n\bigr) \le a_1 n\e_n^2.$$
This shows that
\begin{align*}
\Pi\bigl(f:\|Af - Af_0\|<\e_n\bigr) 
&= \int_0^\infty \Pi_\t(f:\|Af - Af_0\|< \e_n) \,dQ(\t)\\
&\geq e^{- a_1n\e_n^2} Q(\t_n,2\t_n).
\end{align*}
If $\a-d/2<\b$, then $\t_n\ra0$, and \Cref{cond:whitenoise_gaussian_mixture} on $Q$ 
gives that
$$-\log Q(\t_n,2\t_n)\lesssim \t_n^{-2}=n^{(2\b-2\a+d)/(2\b+2\g+d)}\le n^{d/(2\b+2\g+d)}=n\e_n^2,$$
if $\b\le \a$.
If $0<\b<\a-d/2$, then $\t_n\ra\infty$, and \Cref{cond:whitenoise_gaussian_mixture} on $Q$ 
gives that
\begin{align*}
-\log Q(\t_n,2\t_n)\lesssim \t_n^{d/(\a-d/2)}&=n^{(d(\a-d/2-\b)/(\a-d/2)(2\b+2\g+d))} \\
&\le n^{d/(2\b+2\g+d)}=n\e_n^2.
\end{align*}
Finally if $\a-d/2=\b$, then $\t_n=1$ and $Q(\t_n,2\t_n)\gtrsim 1$.
Thus in all three cases $Q(\t_n,2\t_n)$ is bounded below by a power of $e^{-n\e_n^2}$.
Combining this with the preceding, we see that 
$\Pi\bigl(f:\|Af - Af_0\|\leq \e_n\bigr) \ge e^{- a_2n\e_n^2}$, for some positive constant $a_2$, which we can take
bigger than $1$. Then \cref{EqPriorMass} is satisfied for $\e_n$ equal to $\sqrt {a_2}$ times the current $\e_n$.

Verification of \cref{EqPriorMassRough}.
\Cref{LemmaPriorConcentrationRough} gives that 
$$\Pi_\t\bigl(f: \|f-f_0\|_0<2\eta_{n,\t}\bigr)\le \Pi_\t\bigl(f: \|f\|_0<2\eta_{n,\t}\bigr)
\le e^{-a_3(\t/\eta_{n,\t})^{d/(\a-d/2)}},$$
for some constant $a_3$. This is bounded above by $e^{-4a_2n\e_n^2}$ if 
$$\eta_{n,\t}= 2a_4 \t\, n^{(d/2-\a)/(2\b+2\g+d)}=2 a_4 \t\, \eta_n/\t_n,$$
for a sufficiently small constant $a_4>0$.

Verification of \cref{EqPriorApproximationExtended}.
Choosing $x=a_4\eta_n/\t_n=\eta_{n,\t}/(2\t)$ in \Cref{lem:whitenoise_gaussian_mixture_prior_approximation},
we see that the left side of \cref{EqPriorApproximationExtended} is bounded above
by $e^{-4a_2n\e_n^2}$ if $j_n$ satisfies
$$aj_n^{1/2-\a/d}\le a_4\eta_n/\t_n,\qquad\text{ and }\qquad 
ba_4^2(\eta_n/\t_n)^2j_n^{2\a/d}\ge 4a_2n\e_n^2.$$
Both inequalities become equalities for $j_n$ of the order $j_n\simeq  n^{d/(2\b+2\g+d)}$, as indicated
at the beginning of the proof. Since $1/2-\a/d<0$ and $2\a/d>0$, the left side of the
first inequality is decreasing in $j_n$ and the left side of second inequality is increasing.
Thus both inequalities are satisfied for $j_n=a_5 n^{d/(2\b+2\g+d)}$ and a sufficienty large
constant $a_5$. 

Finally we choose $\e_n$ and $j_n$ in \Cref{TheoremGeneralAdaptation} equal to $\sqrt{a_2}$ and
$a_5$ times the orders indicated at the beginning of the proof. Then \cref{EqjEpsilon}
is satisfied, and \cref{EqEtaEpsilonDelta} and \cref{EqEtaDelta} are satisfied if $\eta_n$ is chosen
of the indicated order times a sufficiently large constant.

% % % % % % % % % % % % % % % % % % % % % % % % % % % % % % % % % % % % %

\begin{appendix}
	\crefalias{section}{appsec}
	\crefalias{subsection}{appsubsec}

	\section{Galerkin Projection}
	\label{sec:whitenoise_galerkin_method}
	In this section we collect some (well known) results on the Galerkin method.
	Consider a scale of smoothness classes $(H_s)_{s\in\RR}$ as in \Cref{cond:smoothness_class_structure}.
	
	\begin{lemma}%[Direct and Inverse Inequalities]
		\label{lem:whitenoise_direct_inverse_estimates}
		If $V_j$ is a finite-dimensional space as in \Cref{AssumptionApproximation} such that 
		\cref{eq:approximation_property} and  \cref{eq:stability_property} hold,
		then, for $P_j: H_0\to V_j$ the orthogonal projection onto $V_j$, and $0\le s,t< S$,
		\begin{align}
		\label{eq:whitenoise_projection_estimate_-}
		\|f - P_j f\|_{-t} &\lesssim \d(j, t)\d(j,s) \|f\|_s,\qquad f\in H_0,\\
		\label{eq:whitenoise_inverse_estimate_-}
		\|g\|_s &\lesssim \frac{1}{\d(j, s)\d(j, t)} \|g\|_{-t},\qquad  g\in V_j.
		\end{align}
	\end{lemma}
	
	\begin{proof}
		By the dual norm relation in (ii) of \Cref{cond:smoothness_class_structure},
		and the orthogonality of $f-P_jf$  to $V_j$,
		\begin{align*}
		\|f-P_jf\|_{-t}&=\sup_{\|g\|_t\le 1}\langle f-P_jf,g\rangle_0
		=\sup_{\|g\|_t\le 1}\langle f-P_jf,g-P_jg\rangle_0\\
		&\le \|f-P_jf\|_0 \,\sup_{\|g\|_t\le 1}\,\|g-P_jg\|_0,
		\end{align*}
		by the Cauchy-Schwarz inequality. Here $\|f-P_jf\|_0\lesssim \d(j,s)\|f\|_s$ 
		and $\|g-P_jg\|_0\lesssim \d(j,t)\|g\|_t$, 
		both by \cref{eq:approximation_property}. Inequality \cref{eq:whitenoise_projection_estimate_-} follows.
		
		For the second inequality we have, for $g\in V_j$,
		$$\|g\|_0=\sup_{f\in V_j: \|f\|_0\le 1}\langle g,f\rangle_0
		\lesssim 
		\sup_{f\in V_j: \|f\|_0\le 1}\|g\|_{-t}\|f\|_t,$$
		again by the dual norm relation. Here we can bound $\|f\|_{t}$ by $\|f\|_0/\d(j,t)$, with the
		help of \cref{eq:stability_property}. We obtain \cref{eq:whitenoise_inverse_estimate_-}
		by first bounding $\|g\|_{s}$ with the help of \cref{eq:stability_property} and next using the
		preceding display.
	\end{proof}
	
	Let $A: H \to G$ be an injective bounded operator between separable Hilbert spaces, 
	and let $V_j$ be a finite-dimensional subspace of $H$.
	The Galerkin solution $f^{(j)}\in V_j$ to the image $Af$ of an element $f$ is defined 
	(also see \Cref{sec:whitenoise_general_contraction}) as the element in $V_j$ such that
	$Af^{(j)}$ is equal to the orthogonal projection of $Af$ onto the image space $W_j=AV_j$. 
	Thus, if $Q_j: G\to W_j$ denotes the orthogonal projection onto $W_j$, then the Galerkin
	solution can be written as 
	$$f^{(j)}=R_jA f, \quad \text {for } \qquad R_j=A^{-1}Q_j,$$
	where the inverse $A^{-1}$ is well defined on the linear subspace $W_j$.
	
	If the operators $R_j A$ are uniformly bounded with respect to $j$, then the 
	convergence rate $\|f^{(j)}-f\|_0$ of the Galerkin solution to $f$ is known to be of the same order 
	as the distance $\|P_jf-f\|_0$ of $f$ to its projection on $V_j$.
	(See Section 3.2 and Theorem~3.7 in \cite{kirsch2011introduction}, or the proof below.) 
	In particular, if $f\in H_s$ and $V_j$ satisfies \cref{eq:approximation_property},
	then the convergence rate is given by $\d(j,s)$. 
	
	In order to control the stochastic noise term $\xi$ 
	in the observation scheme \cref{eq:whitenoise_Y=Af+xi}, it is necessary also to control the norms
	of the operators $R_j$. The following lemma summarizes the properties of 
	the Galerkin projection needed in the proof of our main result. 
	
	\begin{lemma}
		\label{lem:whitenoise_R_j_estimates}
		If $V_j$ is a finite-dimensional space as in \Cref{AssumptionApproximation} such that 
		\cref{eq:approximation_property} and  \cref{eq:stability_property} hold,
		and $A: H_0\to G$ is a bounded linear operator satisfying $\|Af\| \simeq \|f\|_{-\g}$ for every $f\in H_0$, 
		then the norms of the operators $R_j: G\to H_0$ and $R_j A: H_0\to H_0$ satisfy
		\begin{align}
		\label{eq:whitenoise_galerkin_R_j_estimate}
		\|R_j\| &\lesssim_A \frac{1}{\d(j,\g)},\\
		\|R_j A\| &\lesssim_A 1.
		\label{eq:whitenoise_galerkin_R_jA_estimate}
		\end{align}
		Furthermore, for $f\in H_s$ the Galerkin solution $f^{(j)}\in V_j$ to $Af$ satisfies
		\begin{align}
		\label{eq:whitenoise_galerkin_solution_error_estimate}
		\|f^{(j)} - f\|_0 \lesssim_A \d(j,s)\, \|f\|_s.
		\end{align}
	\end{lemma}
	
	\begin{proof}
		For $g\in G$ we have $R_jg\in V_j$ and hence by \cref{eq:whitenoise_inverse_estimate_-},
		\begin{align*}
		\|R_j g\|_0 \lesssim \frac1{\d(j,\g)}\|R_jg\|_{-\g}
		\simeq \frac1{\d(j,\g)}\|AR_jg\|
		=\frac1{\d(j,\g)}\|Q_jg\|,
		\end{align*}
		since $AR_j=Q_j$. Because $\|Q_jg\|\le \|g\|$, we conclude that 
		$\|R_j\|\lesssim 1/\d(j,\g)$.
		
		By definition $f^{(j)}=R_jAf$, and $R_jA$ acts as the identity on $V_j$.
		Therefore $f^{(j)}-P_jf= R_jA(f-P_jf)$, and hence
		$$\|f^{(j)}-P_jf \|_0\le \|R_j\|\, \|A(f-P_jf)\|
		\simeq \|R_j\|\, \|f-P_jf\|_{-\g}\le \|R_j\|\, \d(j,\g)\|f\|_0,$$
		by \cref{eq:whitenoise_projection_estimate_-}. By the preceding paragraph $\|R_j\|\, \d(j,\g) \lesssim 1$, so that the right side
		is bounded above by a multiple of $\|f\|_0$. By the triangle inequality 
		$$\|R_jAf\|_0=\|f^{(j)}\|_0\le \|f^{(j)}-P_jf \|_0+\|P_jf-f \|_0+\|f\|_0\lesssim \|f\|_0,$$
		in view of the preceding display and the fact that $\|P_jf-f \|_0\le \|f\|_0$.
		This shows that $\|R_jA\|\lesssim 1$.
		
		Finally, since $f^{(j)}-f= (R_jA-I)(f-P_jf)$, we have that
		$$\|f^{(j)}-f\|_0=\|(R_jA-I)(f-P_jf)\|_0\le \bigl(\|R_jA\|+1\bigr)\|f-P_jf\|_0.$$
		Inequality \cref{eq:whitenoise_galerkin_solution_error_estimate} follows by the boundedness of $\|R_jA\|$
		and \cref{eq:approximation_property}.
	\end{proof}
	
	As is clear from the proof, the smoothing assumption $\|Af\|\simeq \|f\|_{-\g}$ can be  relaxed to the pair
	of inequalities
	\begin{align}
	\label{EqSmoothingConditionRelaxed}
	\|A f\|&\lesssim \|f\|_{-\gamma},\qquad f\perp V_j,\\
	\|A f\|&\gtrsim \|f\|_{-\gamma},\qquad f\in R(R_j).
	\label{EqSmoothingConditionRelaxedtwo}
	\end{align}
	This helps to cover cases in which the smoothing  condition is satisfied for a modification of the
	operator $A$, but not $A$ itself, for example a modification taking different boundary conditions
	of a differential operator into account.
	
	We introduce a \emph{modified Galerkin solution} to $Af$ to cover such a case.
	Let $A_0, A: H \to G$ be injective bounded operators between separable Hilbert spaces that possess a common inverse
	in the sense of existence of a linear map $B: D(B)\subset G\to H$ with domain $D(B)$ containing the linear span of the
	ranges of $A_0$ and $A$ such that that $BA_0=I=BA$. For simplicity of notation, write $B=A^{-}=A_0^{-}$. Intuitively,
	for the inverse problem, taking $A_0f$ or $Af$ as input data should be equivalent. However, it may be that $A_0$ is smoothing
	in a given scale $(H_s)_{s\in\RR}$, whereas $A$ is not. In that case we reconstruct as follows. Assume that
	$\Phi=A-A_0$ has closed range, and let $P_\Phi: G\to G$ be the orthogonal projection onto this range. Now let $Q_j: G\to G$ be the
	orthogonal projection onto the finite-dimensional space $(I-P_\Phi)AV_j$, and set
	\begin{equation}
	\label{EqmodifiedGalerkin}
	f^{(j)}= R_j Af, \quad \text {for } \qquad R_j=A^{-}Q_j(I-P_\Phi).
	\end{equation}
	Thus after removing the ``irrelevant part'' of $Af$ that does not influence the inversion, we project onto the
	finite-dimensional space $(I-P_\Phi)AV_j$ of similarly cleaned functions $Af$ with $f\in V_j$, and finally invert.
	%The inversion is well defined as $(I-P_\Phi)AV_j$ is contained in the span of the ranges of $A$ and $A_0$.
	
	\begin{lemma}
		\label{lem:whitenoise_R_j_estimates_modified}
		If $V_j$ is a finite-dimensional space as in \Cref{AssumptionApproximation} such that 
		\cref{eq:approximation_property} and  \cref{eq:stability_property} hold,
		and $A_0,A: H_0\to G$ are bounded linear operators with common inverse satisfying $\|A_0f\| \simeq \|f\|_{-\g}$ for every $f\in H_0$, 
		then the operators $R_j: G\to H_0$ and $R_j A: H_0\to H_0$  and $f^{(j)}=R_jAf$ as in \eqref{EqmodifiedGalerkin}
		satisfy \eqref{eq:whitenoise_galerkin_R_j_estimate}, \eqref{eq:whitenoise_galerkin_R_jA_estimate} and
		\eqref{eq:whitenoise_galerkin_solution_error_estimate}.
	\end{lemma}
	
	\begin{proof}
		The operator $[A]: H\to G/\Phi(H)$ mapping $f\in H$ into the class of $Af$ in the quotient space $G/\Phi(H)$
		is one-to-one, since $[Af]=0$ implies $Af\in R(\Phi)$ and hence $f=BAf=0$, since $B\Phi=0$. Identifying
		$[g]\in \tilde G:=G/\Phi(H)$ with the function $(I-P_\Phi)g$ with norm $\|[g]\|_{\tilde G}=\|(I-P_\Phi)g\|_G$,
		we see that $R_jAf$ as in \eqref{EqmodifiedGalerkin} is actually the Galerkin solution to $[A]f$. It suffices to
		show that $[A]: H\to \tilde G$ is smoothing in the sense of \eqref{EqSmoothingConditionRelaxed}.
		Now $\|[Af]\|_{\tilde G}=\|(I-P_\Phi)A_0f\|_G\le \|A_0f\|_G\simeq \|f\|_{-\gamma}$, for every $f\in H$.
		Furthermore, for every $f$ such that $A_0f\perp R(\Phi)$, the inequality is an equality. This is
		true for $f=R_jg$, since $A_0R_jg=Q_j(I-P_\Phi)g\in (I-P_{\Phi})AV_j$.
	\end{proof}

	\section{Approximation Numbers and Metric Entropy}
	\label{sec:whitenoise_approximation}
	\label{sec:whitenoise_entropy}
	The $j$th \emph{approximation number} of a bounded linear operator $T: G\to H$ between
	normed spaces is defined as 
	\begin{align}
	\label{EqAppromationNumber}
	a_j(T :G \to H) =\inf_{U: \text{Rank} U < j} \|T - U\|_{G\to H},
	\end{align}
	where the infimum is taken over all linear operators $U: G \to H$ of rank (i.e., dimension of the range
	space) strictly less than $j$, and the norm on the right is the operator norm
	$\|T-U\|_{G\to H}=\sup_{f: \|f\|_ G \leq 1} \|(T - U)f\|_H$.
	The approximation numbers measure the possibility of approximating an operator
	by simpler operators of finite-dimensional rank. There is a rich literature on approximation
	numbers. The main purpose of the present section is to note their relationship to
	singular values and to  metric entropy. Metric entropy plays an important role
	in the characterization of contraction rates of Bayesian posterior distributions.
	
	If $G\subset H$, we can take $T$ equal to the embedding $\iota: G\to H$, and then
	by linearity we see that there exists an operator $U$ of rank smaller than $j$ such that 
	$$\|f - Uf\|_{H}\lesssim a_j(\iota :G \to H) \, \|f\|_G,\qquad \forall f\in G.$$
	If $H$ is a Hilbert space, then the minimizing finite-rank operator $U$ is of course the orthogonal projection $P_j$ on $V_j$.
	However, the approximation numbers also `search' an optimal  projection space.
	If we take $G=H_s$ and $H=H_0$, then the range space $V_j$ of $U$ satisfies the approximation property
	\cref{eq:approximation_property}, with the numbers $\d(j,s)$ taken equal to 
	the approximation numbers $a_j(\iota : H_s \to H_0)$.
	
	The approximation number is an example of an \textit{s-number}, as introduced in
	\cite{pietsch1974_s-numbers}.  In general $s$-numbers are defined as maps $T\mapsto \bigl(s_j(T)\bigr)_{j\in\NN}$, 
	attaching to every operator $T$ a sequence of nonnegative
	numbers $s_j(T)$, satisfying certain axiomatic properties. In general, approximation numbers attached to operators
	$T: H\to H$ are the `largest' possible $s$-numbers, but on Hilbert spaces there is only one $s$-number:
	all $s$-numbers are the same (see 2.11.9 in \cite{pietsch1987eigenvalues}). Because
	the singular values are also $s$-numbers, the latter unicity yields the important relation that
	the approximation numbers of operators on Hilbert spaces are equal to their
	singular values. Recall here that the singular values of a compact operator $T: G\to H$
	are the roots of the eigenvalues of the self-adjoint operator $T^*T: G\to G$.
	
	The finite-rank approximations $U$ that (nearly) achieve the infimum in the definition of the approximation 
	numbers for different $j$ are not a-priori ordered. However, in many
	cases there exists a basis $(\phi_i)_{i\in\NN}$ such that the projections
	on the linear span of the first $j-1$ basis elements achieve the infimum.
	For Sobolev spaces e.g.\,  spline bases, the Fourier basis, or wavelet bases are
	all `optimal' in this sense (see \cite{quarteroni2009numerical_approximation_PDEs,cohen2003numerical_wavelet}).
	
	Approximation numbers are strongly connected to metric entropy. 
	In the literature the connection is usually made through the notion of 
	`entropy numbers', which are defined as follows. 
	The $j$-th \textit{entropy number} $e_j(T)$ of an operator $T: G\to H$ is defined as 
	the infimum of the numbers $\e>0$ so that the image $T(U_G)\subset H$ of the unit ball
	$U_G$ in $G$ can be covered by $2^{j-1}$ balls of radius $\e$ in $H$; or more formally,
	with $U_H$ the unit ball in $H$,
	\begin{align*}
	e_j(T)=&\inf \Bigl\{ \e>0: T(U_G) \subset \bigcup_{i=1}^{2^j - 1} (h_i + \e U_H), 
	\text{for some } h_1,\ldots,h_{2^j-1} \in H \Bigr\}.
	\end{align*}
	The function $j\mapsto e_j(T)$ is roughly the inverse function of the metric entropy
	of $T(U_G)$ relative to the metric induced by $\|\cdot\|_H$. Recall that the \textit{metric entropy} of a metric
	space $(U,d)$ is the logarithm of the covering number $N(\e, U,d)$, which is the minimal number of $d$-balls of
	radius $\e>0$ needed to cover the space $U$. Presently we consider the metric entropy
	$H(\e, T)=\log N\bigl(\e, T(U_G), \|\cdot\|_H\bigr)$ of $T(U_G)$ under the metric of $H$.
	Roughly we have that 
	$$N\bigl(\e, T(U_G), \|\cdot\|_H\bigr)\simeq 2^{j-1},\qquad\text{if}\qquad e_j(T)\simeq\e.$$
	If we use the logarithm at base $2$, then the map $\e\mapsto H(\e,T)$ is approximately inverse to 
	the map $j\mapsto e_j(T)$.
	
	Now it is proved in \cite{edmunds1988entropy_approximation} that
	for any operator $T:G\to H$ between Hilbert spaces with infinite-dimensional ranges:
	$$e_{j+1}(T)\le 2 a_{J+1}(T)\le 2\sqrt 2 e_{J+2}(T), $$
	for any natural numbers $j,J$ satisfying:
	$$j\log 2\ge 2\sum_{i=1}^J\log\frac{3a_i(T)}{a_{J+1}(T)}.$$
	As shown in \cite{edmunds1988entropy_approximation} this relationship between entropy numbers
	and approximation numbers may be
	solved to derive the entropy number from the approximation numbers in many cases.
	
	The following lemma gives one example, important to the present paper.
	
	\begin{lemma}[Metric entropy]
		\label{lem:whitenoise_metric_entropy_hilbert_scales}
		For a smoothness scale $(H_s)_{s\in\RR}$ satisfying \cref{eq:approximation_property} 
		with $\d(j,s)= j^{-s/d}$, and $s>0$ and $t\ge0$,
		\begin{align}
		\label{eq:whitenoise_metric_entropy_bound}
		\log N\bigl(\e, \{f\in H_s: \|f\|_s\le 1\},   \|\cdot\|_{-t}\bigr) \sim \e^{-d/(s+t)}.
		\end{align}
	\end{lemma}
	
	\begin{proof}
		By \cref{eq:whitenoise_projection_estimate_-} the 
		approximation number $a_j(\iota: H_s\to H_{-t})$ is of the order $\d(j,s)\d(j,t)=j^{-(s+t)/d}$. 
		It is shown in \cite{edmunds1988entropy_approximation} that
		the entropy numbers $e_j(\iota: H_s\to H_{-t})$ are of the order $j^{-(s+t)/d}$. 
		By the preceding reasoning  this can be inverted to obtain the order
		of the  metric entropy of the image of the unit ball in $H_{-t}$.
	\end{proof}
	
	In a similar way it is possible to invert approximation numbers that are not of 
	the polynomial form $j^{-s/d}$. There are many examples of this type,
	for instance, involving additional logarithmic terms, or exponentially decreasing
	rates.
\end{appendix}

\bibliographystyle{apa-good}
\bibliography{bib}

\end{document}